\documentclass[a4paper,11pt]{amsart}
\usepackage{algorithm, algpseudocode, amssymb, amsthm, bm, enumerate, hyperref, mathtools, subcaption, svg, tikz, tikz-cd, xcolor}
\usepackage[paper=a4paper,                     % Papierformat
            left=27mm,                         % Seitenrand links
            right=27mm,                        % Seitenrand rechts
            top=25mm,                          % Seitenrand oben
            bottom=25mm,                       % Seitenrand unten
            bindingoffset=0mm                  % Hinzufuegen eines Binderands
           ]{geometry}

\theoremstyle{plain}
\newtheorem{theorem}{Theorem}[section]
\newtheorem{conjecture}[theorem]{Conjecture}
\newtheorem{corollary}[theorem]{Corollary}
\newtheorem{lemma}[theorem]{Lemma}
\newtheorem{proposition}[theorem]{Proposition}

\theoremstyle{definition}
\newtheorem{definition}[theorem]{Definition}
\newtheorem{notation}[theorem]{Notation}

\theoremstyle{remark}
\newtheorem{example}[theorem]{Example}

\newtheorem{question}[theorem]{Question}
\newtheorem{remark}[theorem]{Remark}

\DeclareMathOperator{\NN}{\mathbb{N}}
\DeclareMathOperator{\ZZ}{\mathbb{Z}}

\DeclareMathOperator{\RR}{\mathbb{R}}
\DeclareMathOperator{\CC}{\mathbb{C}}

\DeclareMathOperator{\PP}{\mathbb{P}}

\DeclareMathOperator{\cO}{\mathcal{O}}

\DeclareMathOperator{\Pic}{Pic}

\DeclareMathOperator{\Cone}{Cone}
\DeclareMathOperator{\Conv}{Conv}

\DeclareMathOperator{\lcm}{lcm}
\DeclareMathOperator{\LPE}{LPE}
\DeclareMathOperator{\mnorm}{mnorm}
\DeclareMathOperator{\munorm}{\mu_{\text norm}}
\DeclareMathOperator{\muva}{\mu_{\text va}}
\DeclareMathOperator{\norm}{norm}

\DeclareMathOperator{\pr}{pr}
\DeclareMathOperator{\Proj}{Proj}

\DeclareMathOperator{\RelInt}{RelInt}

\makeatletter
\def\@tocline#1#2#3#4#5#6#7{\relax
  \ifnum #1>\c@tocdepth
  \else
    \par \addpenalty\@secpenalty\addvspace{#2}
    \begingroup \hyphenpenalty\@M
    \@ifempty{#4}{
      \@tempdima\csname r@tocindent\number#1\endcsname\relax
    }{
      \@tempdima#4\relax
    }
    \parindent\z@ \leftskip#3\relax \advance\leftskip\@tempdima\relax
    \rightskip\@pnumwidth plus4em \parfillskip-\@pnumwidth
    #5\leavevmode\hskip-\@tempdima
      \ifcase #1
       \or\or \hskip 1em \or \hskip 2em \else \hskip 3em \fi
      #6\nobreak\relax
    \dotfill\hbox to\@pnumwidth{\@tocpagenum{#7}}\par
    \nobreak
    \endgroup
  \fi}
\makeatother
\definecolor{pastel-pink}{HTML}{ffb3d9}
\definecolor{soft-blue}{HTML}{5eb3f6}
\definecolor{mint-green}{HTML}{66d9a6}
\definecolor{peach}{HTML}{ffac70}
\definecolor{lavender}{HTML}{b39ddb}
\definecolor{coral}{HTML}{ff6b8a}
\definecolor{yellow}{HTML}{ffdb58}

\title{Embeddings of weighted projective spaces}
\author{Praise Adeyemo}
\author{Dominic Bunnett}
\author{Fabián Levicán-Santibáñez}
\address[Adeyemo]{University of Ibadan, Faculty of Science}
\email{ph.adeyemo@ui.edu.ng}
\address[Bunnett]{Technische Universität Berlin, Institut für Mathematik}
\email{bunnett@math.tu-berlin.de}
\address[Levicán-Santibáñez]{Universität Wien, Fakultät für Mathematik}
\email{fabian.levican@univie.ac.at}
\date{October 7, 2025}

\begin{document}

\begin{abstract}
    Let $X$ be a projective toric variety of dimension $n$ and let $L$ be a ample line bundle on $X$.
    For $k \geq 0$, it is in general difficult to determine whether $L^{\otimes k}$ is very ample and whether it additionally gives a projectively normal embedding.
    These two properties are equivalent to the \emph{very ampleness}, respectively \emph{normality}, of the corresponding polytope. By a result of Ewald-Wessels, both statements are classically known to hold for $k \geq n - 1$.

    We study embeddings of weighted projective spaces $\mathbb{P}(a_0, \ldots, a_n)$ via their corresponding rectangular simplices $\Delta(\lambda_1, \ldots, \lambda_n)$.
    We give multiple criteria (depending on arithmetic properties of the weights $a_i$) to obtain bounds for the power $k$ which are sharp in many cases.
    We also introduce combinatorial tools that allow us to systematically construct families exhibiting extremal behaviour.
    These results extend earlier work of Payne, Hering and Bruns-Gubeladze.
\end{abstract}

\maketitle

\tableofcontents

\section{Introduction}

The question of determining which powers of an ample line bundle are very ample is a difficult one in algebraic geometry.
The problem is related to Fujita’s conjecture, where the same question is asked but after twisting by the canonical bundle.

More generally, projective embeddings raise a host of natural questions.
Given a projective variety together with an ample line bundle, one may ask: for which power of the line bundle is the embedding  projectively normal?
When is its image defined by quadrics?
Both properties are known to hold asymptotically, yet sharp bounds are delicate to obtain.
These are the first of the so-called Green-Lazarsfeld conditions, $N_0$ and $N_1$.
In this paper we are concerned with $N_0$ (projective normality) and very ampleness.

To formalise this, we introduce two invariants.
For an ample line bundle $L$ on a projective scheme $X$, we define the normality index $\munorm(X,L)$ as the smallest integer such that for every $m\geq \munorm(X,L)$ the tensor power $L^{\otimes m}$ is very ample and defines a projectively normal embedding.
Similarly, the very ample index $\muva(X,L)$ records the smallest exponent for which the bundle is very ample.

In this paper we investigate these indices in the specific case of weighted projective spaces.
Since the Picard rank is one, we can talk about the index of the weighted projective space itself, where the line bundle in question is taken to be the ample generator of the Picard group, denoted $\muva(\PP(a_0, \dots , a_n))$ and $\munorm(\PP(a_0, \dots , a_n))$.

Moreover, using toric geometry, these questions are equivalent to combinatorial problems about lattice polytopes, which have been the focus of intensive study in discrete geometry and are important in applications such as integer programming.
In the case of weighted projective spaces, the relevant polytopes are simplices, whose discrete geometry encodes the embedding properties of the line bundle.

There is a classical general bound due to Ewald and Wessels: for toric varieties, normality (and therefore very ampleness) can be guaranteed once the degree of the line bundle is larger than one less that the dimension.
Concretely, this says that for $\PP = \PP(a_0, \dots , a_n)$ we have that $\muva(\PP) \leq \munorm(\PP) \leq n-1$.

For $k \in \ZZ_+$ and a polytope $P \subset \RR^n$, we say that it satisfies the \emph{$\mathit{k}$ lattice points on edges} or $\LPE(k)$ property if each edge of $P$ contains at least $k$ lattice points.
Fujita’s conjecture for singular toric varieties was proven by Payne \cite{Payne:2006}. Payne’s argument shows that if a simplex satisfies the $\LPE(n)$ property, then the corresponding line bundle is very ample.
For completeness, and to set the stage, we provide a streamlined version of Payne’s proof in Section~\ref{section-general-results}.
Later, in \cite{Gubeladze:2012}, Gubeladze proved that if a polytope satisfies the $\LPE(4n(n+ 1))$ property, then the corresponding embedding is projectively normal.

Using Payne's criterion and another simple result we get the following for a well-formed projective space $\PP = \PP(a_0,\dots , a_n)$:

\begin{theorem}(Corollary \ref{corollary:distinct-weights})
    Suppose that for every $a_i \neq a_j$ we have $\gcd(a_i,a_j)=1$.
    Then $\muva(\PP) = 1$.
    That is, every ample line bundle is very ample.
\end{theorem}

This mirrors the situation for smooth toric varieties.
Note that we do not say whether the corresponding embedding is projectively normal.

The next result says that maximising very ampleness (and thus normality) indices requires assuming distinct weights.
\begin{theorem}(Corollary \ref{corollary:max-muva-distinct-weights})
        If $\muva(\PP) = n-1$ then $a_i \neq a_j$ for $i \neq j$.
        The same is true for $\munorm$.
\end{theorem}

The first example given in the literature (and only, until this paper and \cite{MullerPaemurru:2025}) of a weighted projective space with a strictly ample line bundle was $\PP(1,6,10,15)$ was found by Ogata \cite{Ogata:2005}.
We prove that this is, in a precise way, the 'first' example and that the arithmetic structure of the weights $(6,10,15)$ is in fact archetypal for weighted projective spaces of dimension 3.

This example shows that the Ewald-Wessels upper bound is sharp for weighted projective 3-spaces and while they provide a general construction of a simplex with maximal normality index, their example gives a so-called \emph{fake} weighted projective space.
We provide an algorithmic construction of a genuine \emph{real} weighted projective space with maximal normality index.

In \cite{EwaldWessels:1991}, they prove their result by first translating the problem into one about polytopes.
Recall that a pair $(X,L)$, where $X$ is a projective toric variety and $L$ is a ample line bundle, corresponds, up to isomorphism, to a polytope $P \subset \RR^n$.

The case of a weighted projective space \emph{with at least one weight is equal to one} and an ample line bundle corresponds to that of \emph{rectangular simplices} (see Subsection \ref{subsection:weighted-projective-spaces} for more details). For the remainder of the paper, we restrict our study to this family of simplices. Properties of a weighted projective space and an ample line bundle can be translated to properties of the corresponding simplex and vice versa, so we will often give statements only in terms of the objects on one side of this correspondence.

\begin{definition}\label{definition:rectangular-simplex}
    Let $\lambda = (\lambda_1, \lambda_2, \ldots, \lambda_n) \in \ZZ_+^n$ be an $n$-tuple of positive integers. The \textbf{rectangular simplex} $\Delta(\lambda)$ corresponding to $\lambda$ is the convex hull of the set
    \[
    \{0, \lambda_1e_1, \lambda_2e_2, \ldots, \lambda_ne_n\},
    \]
    where $\{e_i\}$ is the canonical basis of $\RR^n$.
\end{definition}

In \cite{BrunsGubeladze:1999}, Bruns and Gubeladze study the normality and koszulness of these simplices. In particular, they prove that their normality is \emph{periodic} in a precise sense, and introduce weaker notions of normality they call \emph{1-normality} and \emph{almost 1-normality}. We will make heavy use of the latter throughout.

We will see that the $\LPE(k)$ property in rectangular simplices can be characterised in terms of a divisibility condition satisfied by its entries (Proposition \ref{proposition:lpe-characterisation}). In Subsection \ref{subsection:subsequences-and-extensions}, we will introduce two extremal properties of rectangular simplices and partly describe the behaviour of the normality index of a rectangular simplex under taking sub- and supersequences of its entries. Later, we will specialise a result of Hering, Schenck and Smith \cite[Theorem 1.1, Corollary 1.4]{HerinSchenckSmith:2006} to improve that of Ewald and Wessels for rectangular simplices:

\begin{proposition}(Proposition \ref{proposition:milena-bound})
    Let $r \in \ZZ_+$ and define $d(\lambda) \coloneqq \left \lfloor \sum_{i = 1}^n \frac{1}{\lambda_i} \right \rfloor$. If $r \geq n - d(\lambda)$, then $r\Delta(\lambda)$ is normal.
\end{proposition}

In Subsection \ref{subsection:periodicity} we will extend \cite[Theorem 1.6]{BrunsGubeladze:1999} to show that the normality index of a rectangular simplex is periodic and give a few interesting consequences. The main result of Section \ref{section:maximally-non-normal} will be the following:

\begin{theorem}(Theorem \ref{theorem:infinitely-many-maximally-non-normal})
    Let $n \geq 3$ and $P = \{p_1, \ldots, p_{n - 1}\}$ be distinct primes such that $\sum_{i \in [n - 1]}~1/p_i~\leq~1$. Then, there exist infinitely many maximally non-normal rectangular simplices of the form $\Delta(p_1, \ldots, p_{n - 1}, p_n)$ with $p_n \notin P$ prime. Furthermore, for $N \in \ZZ_+$, there exists an algorithm (Algorithm \ref{alg:1}) that outputs $N$ such simplices in finite time.
\end{theorem}

To the best of our knowledge, finding even one example of a maximally non-normal rectangular $n$-simplex was previously open for high values of $n$.

Finally, in Section \ref{section:hypergraphs} we will introduce a very natural correspondence between hypergraphs and rectangular simplices (Theorem \ref{theorem:hypergraphs-rectangular-simplices-correspondence}) that encodes the incidence structure of primes and entries of the rectangular simplex. This will allow us to read the $\LPE(k)$ property in a purely combinatorial way (Proposition \ref{proposition:conditions-not-depending-on-weighting}). It will also allow us to prove almost 1-normality in many cases (Example \ref{example:hypergraphs-lpe-a1n}), by using a connection to the \emph{Frobenius coin problem.} We will also prove that the $\LPE(n)$ property and almost 1-normality cannot be used to find a rectangular simplex that is both very ample and non-normal:

\begin{theorem}(Theorem \ref{theorem:lpe-implies-a1n}) If the rectangular simplex $\Delta(\lambda) = \Delta(\lambda_1, \ldots, \lambda_n)$ has the LPE$(n)$ property, then $\Delta(\lambda)$ is almost 1-normal.
\end{theorem}

Using our methods we can easily deduce concrete results, such as that the rectangular simplex $\Delta(\lambda) = \Delta(2, 5, 7, 11, 619)$ has maximal normality index equal to $4$ (Corollary \ref{corollary:2-5-7-11-619}).
We can also prove the following:

\begin{proposition}(Corollary \ref{corollary:2-3-p})
    Let $p\geq 5$ be a prime number and $\ell = \lcm(2,3,p) = 6p$.
    The line bundle $\cO(\ell)$ on weighted projective space $\PP(1,6,2p,3p)$ is very ample if and only if $p \equiv 1 \mod 3$.
\end{proposition}

Figure \ref{fig:implications} shows the implication graph of all the properties we will discuss in the special case of rectangular simplices. All implications are strict, except perhaps for ``normality (N) implies very ampleness (VA)''. For implications involving the $\LPE(k)$ property, the lower bound given by $k$ is sharp.

After the completion of this work, a preprint by Muller and Paemurru \cite{MullerPaemurru:2025} appeared on the \emph{arXiv}, addressing similar questions and arriving at similar results independently.
More specifically, \cite[Theorem 1.14]{MullerPaemurru:2025} is a special case of Corollary \ref{corollary:distinct-weights}, where no weights are repeated.
Furthermore, the methods they use to prove that the set of weights with a \emph{bubble} is infinite (see \cite[Section 7]{MullerPaemurru:2025}) are very similar to the ones we use in Section \ref{section:maximally-non-normal}, and may also be used to maximise the normality index of a weighted projective space \cite[Theorem 7.5]{MullerPaemurru:2025}.

\begin{notation}
    We will use the following notation throughout the paper:
    
    For a set $S$, $\# S$ is its cardinality and $\mathcal{P}(S)$ is its power set.
    
    $\NN$ is the set of non-negative integers. $\ZZ_+$ is the set of positive integers. $[n]$ is the subset $\{1, 2, \ldots, n\} \subset \ZZ_+$.

    A polytope $P \subset \RR^n$ is a convex polyhedron that is bounded, has all of its vertices in the lattice $\ZZ^n$, and whose linear span is of dimension $n$. In other words, unless specified otherwise, we only consider full-dimensional convex lattice polytopes.
    
    For $\lambda = (\lambda_1, \lambda_2, \ldots, \lambda_n) \in \ZZ_+^n$ and $r \in \ZZ_+$, $r\lambda = (r\lambda_1, r\lambda_2, \ldots, r\lambda_n)$ is the $r$-th dilation of $\lambda$. Similarly, for a polytope $P \subset \RR^n$, $rP = \{x \in \RR^n:~~x/r \in P\}$ is the $r$-th dilation of $P$.
\end{notation}

\begin{figure}[ht]
    \centering
    \begin{tikzpicture}[scale=2]
        % Nodes
        \node (A) at (0,1) {$\lambda = r\lambda',~~r \geq n - 1$};
        \node (B) at (2,1) {N};
        \node (C) at (4,1) {A1N};
        \node (D) at (-2,0) {$\LPE(4n(n + 1))$};
        \node (E) at (0,0) {$\LPE(n)$};
        \node (F) at (2,0) {VA};
    
        % Edges
        \draw[-{Implies}, double, very thick] (A) -- (B);
        \draw[-{Implies}, double, very thick] (B) -- (C);
        \draw[-{Implies}, double, very thick] (D) -- (E);
        \draw[-{Implies}, double, very thick] (E) -- (F);
        \draw[-{Implies}, double, very thick] (A) -- (E);
        \draw[-{Implies}, double, very thick] (B) -- (F);
        
        % Curved edges
        \draw[-{Implies}, double, very thick, bend left=60] (D) to (B);
        \draw[-{Implies}, double, very thick, bend right=60] (E) to (C);
    \end{tikzpicture}
    \caption{Implication graph of properties of rectangular simplices.}
    \label{fig:implications}
\end{figure}
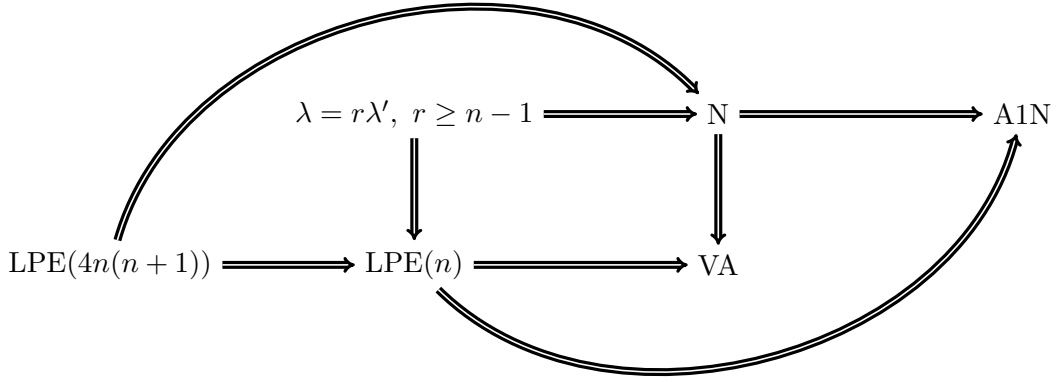

\section*{Acknowledgements}

This work was carried out as part of a “Research in Teams” project at the Erwin Schrödinger International Institute for Mathematics and Physics (ESI) in Vienna, and we gratefully acknowledge the Institute’s support and warm hospitality. 
DB and PA thank the Abram Gannibal Project for supporting a research visit, during which this project was initiated. 
PA would like to thank AFRIMath for their support.
We also thank Balázs Szendrői for his constant encouragement and support, and Winfried Bruns for kindly informing us about the example of a very ample but not normal non-rectangular simplex.

\section{Preliminaries and previous work}

In this section we give basic definitions and provide references to some existing results.

\subsection{Properties of polytopes} \hfill

\begin{definition}\label{definition:normal-very-ample}
    Let $P \subset \RR^n$ be a polytope.
    We say that $P$ is \textbf{very ample} if there exists a $k_0 \in \NN$ such that for every $k \geq k_0$ and every $x \in kP \cap \ZZ^n$ there exists $x_1 , \dots , x_k \in P \cap \ZZ^n$ such that $x = x_1 + \cdots + x_k$.
    If we can choose $k_0 = 1$ then we say that $P$ is \textbf{normal}.
\end{definition}

For each $k \in \NN$, this can be expressed as
\[kP \cap \ZZ^n = P \cap \ZZ^n + \cdots + P \cap \ZZ^n \enspace ,\]
which we sometimes refer to this as the IDP property for $kP$.

There is also a characterisation of very ampleness via the semigroups at each edge.

\begin{definition}
Let $P \subset \RR^n$ be a polytope and $v \in P$ a vertex.
Define the \textbf{vertex semigroup} at $v$ to be $S_{P,v} = \NN(P \cap \ZZ^n - v)$.
\end{definition}

The following is well-known \cite[Exercise 2.23]{BrunsGubeladze:2009}.

\begin{proposition}\label{prop:very-ample-saturated}
A polytope $P \subset \RR^n$ is very ample if and only if for every vertex $v \in P$, the vertex semigroup $S_{P,v}$ is saturated.
\end{proposition}

We define indices associated to the properties from Definition \ref{definition:normal-very-ample}.

\begin{definition}
    Let $P \subset \RR^n$ be a polytope.
    We define
    \[\muva(P) = \min \{k \in \NN \,\, | \,\, kP \text{ is very ample}\} \enspace ,\]
    \[\munorm(P) = \min \{k_0 \in \NN\,\, | \,\, kP \text{  normal } \forall \, k\geq k_0\} \enspace .\]
\end{definition}
These are well-defined, one way to see this is via the correspondence to toric varieties, see Proposition \ref{proposition:toric-polytope-correspondence}.

\begin{remark}
    Note that just because $kP$ is normal does not necessarily mean that $(k+1)P$ is normal.
    Another invariant we may define is the minimal $k_0$ such that $k_0P$ is normal. This is sometimes called the \emph{minimal normality index} $\mu_{\mnorm}$ and can be strictly smaller than our normality index.
    For an example of this phenomena and a study of many other indices and how they relate to one another see \cite{CoxHasseHibiHigashitani:2014}.
\end{remark}

The following result is fundamental for our study and shows that the normality (and therefore very ampleness) index is bounded uniformly above for all polytopes of the same dimension.
It was first stated in the following form in \cite{BrunsGubeladzeTrung:1997}, however the proof goes back to Ewald-Wessels \cite{EwaldWessels:1991} who formulate it in terms of toric varieties.

\begin{theorem}\label{theorem:EW}
    Let $P \subset \RR^n$ be a polytope. Then for every $k\geq n-1$ the dilation $kP$ is normal.
\end{theorem}

\subsection{Toric varieties} \hfill

We recall a few basic notions of toric geometry mostly to fix notation, we use standards as laid out in \cite{CoxLittleSchenck:2011}.

Given a polytope $P \subset \RR^n$ we denote the corresponding pair $(X,L)$, where $X$ is a toric variety and $L$ is an ample reflexive rank one sheaf.
This association works by considering the semigroup $S(P)=\Cone(P \times\{1\}) \cap \ZZ^{n+1}$ and then defining $X = \Proj k[P]$, where $k[P]$ is the $k$-algebra associated to $S[P]$.

There is an association going the other direction (although, as stated one must fix an isomorphism $T \cong (\CC^*)^n$) which assigns a lattice polytope to such an ample pair.
Up to sensible notions of equivalence, this correspondence can be made one-to-one.

Crucially, under this correspondence, when $L$ is a line bundle, then a dilation $kD$ corresponds to $L^{\otimes k}$.

The following two results are well-known, see \cite{CoxLittleSchenck:2011}.

\begin{proposition}\label{proposition:toric-polytope-correspondence}
    Given $P$ and $(X,L)$ as above, then the following hold.
    \begin{enumerate}
        \item $P$ is very ample if and only if $L$ is very ample.
        \item $P$ is normal if and only if $L$ is very ample and the embedding into $\PP(H^0(X,L)^*)$ is projectively normal.
    \end{enumerate}
\end{proposition}

\begin{theorem}
    Suppose that $P$ and $(X,L)$ are as above and suppose further that $X$ is smooth.
    Then $P$ is very ample.
\end{theorem}

Thus polytopes coming from smooth toric varieties are always very ample, however the same statement for normality is a famous conjecture of Oda.

\begin{conjecture}
    Suppose that $P$ and $(X,D)$ are as above and that $X$ is smooth.
    Then $P$ is normal.
\end{conjecture}

These questions are deeply related to the regularity of an embedding and can be expressed cohomologically.
For background and related results, see \cite{Hering:Thesis,HerinSchenckSmith:2006,Oberwolfach:2007}.

\subsection{Weighted projective spaces and simplices}\label{subsection:weighted-projective-spaces} \hfill

If $P$ is a simplex, then the associated ample toric pair $(X,L)$ has Picard (and class) rank 1 and $X$ is a so-called `fake' weighted projective space and $L \cong \cO_X(d)$, some power of the generator of the class group (note that is $L$ is not necessarily a line bundle we have to be a little careful with what precisely we mean by power).

Consider a weight vector $\underline{a} = (a_0, \dots , a_n) \in \ZZ_{>0}^{n+1}$, then denote the corresponding weighted projective space $\PP = \PP(\underline{a}) = \Proj k[x_0, \dots , x_n]$ with the grading given by our weight vector.
We assume that our presentaion of $\PP$ is well-formed, that is, that for every $i = 0 ,\dots , n$ we have $\gcd(a_0, \dots, \hat{a}_i, \dots , a_n) = 1$.

It is well-known that $\PP$ is a toric variety.
Then $\Pic \PP \cong \ZZ$ is generated by $\cO_{\PP}(l)$ where $l = \lcm (a_0, \dots , a_n)$ and we construct the corresponding polytope to $(\PP, \cO_{\PP}(l))$ as the `section polytope' of $\cO_{\PP}(l)$.
That is, we see the monomials of $H^0(\PP,\cO_{\PP}(l) = k[x_0, \dots , x_n]_l$ as lattice points and take their convex hull.

\begin{definition}\label{definition:weighted-simplex}
    Let $\PP = \PP(\underline{a})$ be as above and write $l_i = \frac{l}{a_i}$ for each $i = 0, \dots , n$.
    We define $\Delta_{\underline{a},\,l} = \Conv(l_0e_0 , \dots , l_ne_n) \subset \RR^{n+1}$.
\end{definition}

Then the pair associated to $\Delta_{\underline{a},l}$ is $(\PP, \cO_{\PP}(l))$.
Thus to answer the question ''for which $k$ is $\cO_{\PP}(kl)$ is very ample?", we study simplices of this form.
Futhermore, when at least one weight is $1$, the simplices are isomorphic to rectangular simplices, as defined in Definition \ref{definition:rectangular-simplex}.

\begin{lemma}\label{lemma:projection-to-rec-simplices}
    Consider $(1,a_1, \dots , a_n) \in \ZZ^{n+1}_{>0}$ and $l$ and the $l_i$ as above.
    The projection away from the first coordinate $\pr_0 : \RR^{n+1} \rightarrow \RR^n$ restricts to a bijection
    \[m\Delta_{\underline{a},l} \cap \ZZ^{n+1}\rightarrow m\Delta(l_1 , \dots , l_n) \cap \ZZ^n\]
    for any $m>0$.
    Moreover, this induces a graded isomorphism of semigroups algebras $k[\Delta_{\underline{a},l}] \xrightarrow[]{} k[\Delta(l_1 , \dots , l_n)]$.
\end{lemma}

\begin{proof}
    That $\pr_0 : {\Delta_{\underline{a},l}} \cap \ZZ^{n+1} \rightarrow \Delta(l_1, \dots , l_n) \cap \ZZ^n$ is a bijection is just dehomogenisation when viewing lattice points as monomials.
    The graded isomorphism then follows as homogeneous generators are sent to homogeneous generators.
\end{proof}

\begin{remark}
    The normality of related polytopes is studied in \cite{BraunDavisSolus:2018,BraunDavisHanelySolus:2024}.
    They impose reflexivity (equivalently, that there is exactly one lattice point in the relative interior); we do not make this restriction.
    We note that the reflexive simplices analysed in \cite{BraunDavisSolus:2018,BraunDavisHanelySolus:2024} are rational multiples of the \emph{polar duals} of the polytopes $\Delta$ that give our weighted projective spaces $(X,L)$.
    In particular, even though the discussion there is phrased in terms of these weighted projective spaces, the varieties they directly study correspond to $\Delta^\vee$ rather than to $\Delta$ itself.
    In this sense, the two viewpoints are complementary, but dualising does not, in general, preserve normality or the normality index, so translations between the settings is not generally possible.
    Furthermore, their emphasis is only on normality, that is, they ask whether $\munorm(\Delta)=1$, and do not track indices.
\end{remark}

Note that, a weighted projective space associated to a rectangular simplex is never fake.
Indeed, the lattice generated by the lattice points of a rectangular simplex is the full $\ZZ^n$.

\begin{remark}
    We make two remarks concerning the toric varieties associated to rectangular simplices.
    \begin{enumerate}
        \item Note that, a weighted projective space associated to a rectangular simplex is never fake.
        Indeed, the lattice generated by the lattice points of a rectangular simplex is the full $\ZZ^n$ and not a strict sublattice which characterises fakeness, see for example \cite{Kasprzyk:2009}.
        
        \item The well-formed condition on $\PP(1,a_1, \dots , a_n)$ is equivalent to $\gcd(a_1, \dots , a_n)=1$.
        However, via the projection in Lemma \ref{lemma:projection-to-rec-simplices}, one can show that the simplices $\Delta_{(1,a_1,\dots , a_n),l}$ and $\Delta_{(1,ka_1, \dots ,ka_n),kl}$, and their semigroups are both equivalent to the same rectangular simplex and its semigroup.
    \end{enumerate}
    
\end{remark}

\begin{example}[\emph{The} counterexample]\label{example:Ogata-counterexample}
Consider $\PP = \PP(1,6,10,15)$.
Then $\cO(30)$ generates the Picard group and the associated rectangular simplex is $\Delta = \Delta(2,3,5)$.
However, as first remarked by Ogata \cite{Ogata:2005}, $\Delta$ is not normal, nor is it very ample.
\end{example}

In fact, \cite{Ogata:2005} proves that for 3-dimensional simplices, very ample is equivalent to normal.
One may ask, is this always the case?
However, as shown in \cite[Example 2.3]{CoxHasseHibiHigashitani:2014}, there exists a 7-dimensional non-rectangular simplex which is very ample, but non-normal.
This example corresponds to a strictly fake weighted projective space; one can readily compute the class group to find non-zero torsion.
The question remains open for simplices giving weighted projective spaces and in particular for rectangular simplices.

\subsection{Periodicity and almost 1-normality} \hfill

In \cite{BrunsGubeladze:1999}, Bruns and Gubeladze study normality and koszulness of semigroup algebras corresponding to rectangular simplices. This will be one of our main references and we will extend many of their results. Remarkably, they prove that normality of rectangular simplices is \emph{periodic with period the least common multiple of the entries}:

\begin{theorem}\label{theorem:periodicity}\cite[Theorem 1.6]{BrunsGubeladze:1999}
    Let $\lambda = (\lambda_1, \ldots, \lambda_n) \in \ZZ_+^n$, $i = 1, \ldots, n$ and $\ell_i = \lcm(\lambda_1, \ldots, \lambda_{i - 1}, \lambda_{i + 1}, \ldots, \lambda_n)$. Then, $\Delta(\lambda)$ is normal if and only if $\Delta(\lambda_1, \ldots, \lambda_i + \ell_i, \ldots, \lambda_n)$ is normal.
\end{theorem}

\begin{remark}
    Since the normality of $\Delta(\lambda)$ is (for low values of $n$) easily decidable with software such as \emph{Polymake}, given fixed $\lambda_1, \ldots, \lambda_{n - 1} \in \ZZ_+$, if we check the normality of exactly $\ell_n$ rectangular simplices, we may then decide whether \emph{any} $\Delta(\lambda)$ is normal by computing $\lambda_n \mod \ell_n$. This idea can be extended in many ways, and we will use it for example in Section \ref{section:maximally-non-normal}.
\end{remark}

Bruns and Gubeladze also give several other criteria for normality and koszulness of rectangular simplices in \cite[Proposition 2.1, Proposition 2.4]{BrunsGubeladze:1999}. They also introduce two strictly weaker forms of normality which they call \emph{1-normality} (1N) and \emph{almost 1-normality} (A1N). Since the latter is much more straightforward to check, it will play a central role in our paper, especially in Section \ref{section:hypergraphs}.

\begin{definition}\label{definition:a1n}
    Let $\lambda = (\lambda_1, \ldots, \lambda_n) \in \ZZ_+^n$, $L = \lcm(\lambda_1, \ldots, \lambda_n)$. For $i = 1, \dots, n$, set $L_i = L/\lambda_i$. Let $d = \gcd(L_1, \ldots, L_n)$. We say $\Delta(\lambda)$ is \textbf{almost 1-normal} (A1N) if $L - d$ is in the semigroup generated by the $L_i$.
\end{definition}

\section{General results}\label{section-general-results}

\subsection{Repeated weights} \hfill

We prove that repeated weights affect neither normality nor very ampleness.
This first appeared in \cite[Lemma 2.15]{ColomenarejoGaluppiMichalek:2020}, where the authors prove the result for normality.
We repeat their proof here to show without doubt that the result also holds for very ampleness and for indices.

\begin{proposition}\label{proposition:repeated-weights}
    Consider the weighted projective spaces $\widetilde{\PP} = \PP(a_0, \dots , a_n)$ and $\PP = \PP(a_0, \dots , a_n , a_n)$, where they only differ by dropping the last repeated weight $a_n$.

    Then 
    \[\muva(\PP) = \muva(\widetilde{\PP}) \quad\text{ and }\quad \munorm(\widetilde{\PP}) = \munorm(\PP) \enspace.\]
\end{proposition}

\begin{proof}
    We translate this into a problem of the corresponding simplices.
    Consider an $n$-tuple of positive integers $b_0, \dots , b_n$ and we denote $\Delta = \Conv(b_0e_{0}, \dots , b_ne_n, b_ne_{n+1}) \subset \RR^{n+2}$ and $\widetilde{\Delta} = \Conv(e_0b_0, \dots , e_nb_n) \subset \RR^{n+1}$.
    Note that $\widetilde{\Delta}$ can be seen as a face of $\Delta$.
    We show that $\Delta$ is very ample if and only if $\widetilde{\Delta}$ is very ample and remark that the proof for normality is the exact same.
    
    We define a map $\widetilde{ \cdot} : \RR^{n+2} \rightarrow \RR^{n+1}$ which sends $(x_0, \dots , x_n,x_{n+1}) \mapsto (x_0 + x_1, \dots , x_n+x_{n+1})$.
    This map restricts to a map $\widetilde{\cdot} : \Delta \rightarrow \widetilde{\Delta}$, mapping lattice points to lattice points.
    Moreover, it surjects onto the lattice points of $\widetilde{\Delta}$ and every fibre over such a lattice point contains at least one lattice point of $\Delta$.
    
    In fact, we can describe the lattice points in the fibres explicitly.
    That is, given $(x_0, \dots , x_n) \in \widetilde{\Delta}\cap \ZZ^{n+1}$, the fibre consists of points
    \[\left\{(x_0, \dots ,x_{n-1},x_{n}-i,i)\,\, | \,\, i=0, \dots , x_n \right\}.\]

    One direction is immediate: $\Delta$ being very ample implies $\widetilde{\Delta}$ is very ample, since it is a face of $\Delta$.
    For the converse, suppose that $\widetilde{\Delta}$ is very ample such that for every $k \geq k_0$ we have the IDP condition.
    Fix such a $k \geq k_0$ and take a lattice point $x \in \Delta \cap \ZZ^{n+1}$.
    Then $\widetilde{x} = \widetilde{x}_1 + \cdots + \widetilde{x}_k$ for $x_i \in \widetilde{\Delta} \cap \ZZ^n$.
    Then we can choose the $x_i$ over each $\widetilde{x}_i$ such that
    $x = x_1 + \cdots + x_k$.
    Thus $\Delta$ is very ample.

    It follows that the normal and very ample indices are the same.
\end{proof}

We get immediately the following two results:
\begin{corollary}\label{corollary:max-muva-distinct-weights}
    Let $\PP = \PP(a_0, \dots , a_n)$ such that $\muva(\PP) = n-1$, then $a_i = a_j$ if and only if $i = j$.
    The same holds for $\munorm$.
\end{corollary}

\begin{corollary}\label{corollary:3-distinct-weights-VA}
    Any weighted projective space with no more than 3 distinct weights has very ample index 1.
    Moreover, any embedding in a complete linear system is projectively normal.
\end{corollary}

\begin{remark}
    Note that this proof actually works for any toric variety, only that ``repeated weights'' doesn't have such a concrete description.
    In the general case, we can collapse any two variables in the Cox ring which are equivalent in the class group.
\end{remark}

\subsection{Lattice points on edges} \hfill

\begin{definition}
    Let $k \in \ZZ_+$ and $P \subset \RR^n$ be a polytope. We say that $P$ satisfies the \textbf{$\bm{k}$ lattice points on edges} or $\bm{\LPE(k)}$ property if each edge of $P$ has at least $k$ lattice points.
\end{definition}

Since our polytopes have all of its vertices in the lattice, any $P \subset \RR^n$ satisfies the $\LPE(2)$ property.

\begin{remark}
    We note that the condition for lattice points on edges can be expressed purely in toric geometric terms.
    Given $P$ and corresponding pair $(X,L)$, an edge $E \subset P$ corresponds to a $T$-invariant curve $C \subset X$.
    Then $|E \cap \ZZ^n| = D \cdot C$, where $D$ is an effective divisor such that $L \cong \cO_X(D)$.
\end{remark}

We now prove a result, due initially to Payne \cite{Payne:2006} in a more general context, which states that given a $n$-dimensional polytope $P$, then $\LPE(n)$ implies very ampleness. Our proof is shorter and more concise, although the core idea is the same.

\begin{theorem}\label{theorem:Payne-LPE}
Let $P \subset \RR^n$ be a polytope satisfying the $\LPE(n)$ property. Then $P$ is very ample.
\end{theorem}

\begin{proof}
We shall prove that every vertex semigroup is saturated.
Pick a vertex $v \in P$.
Since $S_{P,v}=S_{P-v,\underline{0}}$ we shall replace $P$ by $P-v$ and assume that $v=\underline{0}$.
Suppose that $E_1, \dots , E_s$ are the edges with $v$ as a vertex and label $f_1, \dots , f_s$ the lattice points of the $E_i$'s closest to $v$.
Define 
\[Q = (n-1) \cdot \Conv(v,f_1, \dots , f_s).\]
Since $|E_i \cap \ZZ^n| \geq n$ we have that $Q\subset P$ and both polytopes $P$ and $Q$ having $v$ as a vertex.
Note that, by construction, $Q$ is normal by Proposition \ref{theorem:EW} and there exists an $l>0$ such that $P \subset l \cdot Q$.
Thus, by the Lemma \ref{lemma:vertex-semigroups} below, we have that the semigroups generated at $v$ are equal, that is $S_{P,v} = S_{Q,v}$.
Thus, since $Q$ is normal and hence very ample, we have that $S_{P,v}$ is saturated.

We repeat this argument at every vertex and conclude via the very ample description of Proposition \ref{prop:very-ample-saturated} that $P$ is very ample.
\end{proof}

\begin{lemma}\label{lemma:vertex-semigroups}
Consider $Q \subset P$, two polytopes, both with the origin $v=\underline{0}$ as a vertex.
Suppose that $Q$ is normal and that there exists an $l\in \NN$ such that $P \subset l \cdot Q$.
Then $S_{P,v} = S_{Q,v}$ and both are saturated.
\end{lemma}

\begin{proof}
Since $Q\cap \ZZ^n \subset P\cap \ZZ^n$, clearly $S_{Q,v} = \NN(Q\cap \ZZ^n) \subset \NN(P \cap \ZZ^n) = S_{P,v}$.

For the opposite inclusion take an arbitrary $p \in P \cap \ZZ^n$.
Then as $P \subset l \cdot Q$ we have that $p \in l \cdot Q$ and by normality of $Q$, there exists $q_1, \dots , q_l \in Q$ such that $p = q_1 + \cdots + q_l$.
Hence $p \in \NN(Q \cap \ZZ^n)$.
Thus $P \cap \ZZ^n \subset \NN(Q \cap \ZZ^n)$ and we conclude
\[\NN(P \cap \ZZ^n) \subset \NN(Q \cap \ZZ^n)\]
as required.

Since $Q$ is normal, it is very ample and thus every vertex semigroup is saturated by Proposition \ref{prop:very-ample-saturated}.
In particular, $S_{Q,v} = S_{P,v}$ is saturated.
\end{proof}

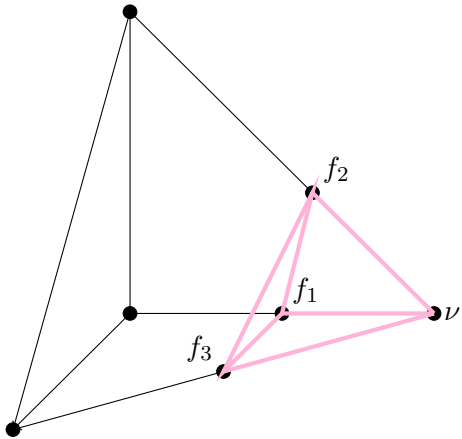
\begin{figure}[!hbt]
\begin{center}
\begin{tikzpicture}
\draw[fill] (0,0,0) circle [radius=0.090];
\draw[fill] (4,0,0) circle [radius=0.090];
\draw[fill] (0,4,0) circle [radius=0.090];
\draw[fill] (0,0,4) circle [radius=0.090];
\draw[fill] (2.4,1.6,0) circle [radius=0.090];
\draw[fill] (2,0,2) circle [radius=0.090];
\draw[fill] (2,0,0) circle [radius=0.090];
\draw (0,4,0)--(0,0,0)--(2,0,0)--(4,0,0)--(2,0,2)--(0,0,4)--(0,4,0)--(2.4,1.6,0)--(4,0,0);
\draw [->] (0,0,0)--(0,0,4);
\draw[ultra thick,pastel-pink] (2,0,0)--(2,0,2)--(2.4,1.6,0)--(2,0,0);
\draw[ultra thick,pastel-pink] (2,0,0)--(4,0,0);
\draw[ultra thick,pastel-pink] (2.4,1.6,0)--(4,0,0);
\draw[ultra thick,pastel-pink] (2,0,2)--(4,0,0);

\node[right] at (4,0,0) {$\nu$};
\node[above right] at (2,0,0) {$f_{1}$};
\node[above right] at (2.4,1.6,0) {$f_{2}$};
\node[above left] at (2,0,2) {$f_{3}$};

\end{tikzpicture}
\end{center}
\caption{The corner polytope $Q$ is shown here in pink.}
\end{figure}

Let us apply this to simplices and weighted projective spaces.

\begin{corollary}\label{corollary:distinct-weights}
Consider a weighted projective space $\PP = \PP(a_0, \dots , a_n)$.
If all distinct weights are pairwise coprime, then $\munorm(\PP) = 1$.
\end{corollary}

\begin{proof}
By Proposition \ref{proposition:repeated-weights} we can assume that all the weights are distinct.
We must show that every edge has $n$ lattice points.
We can identify the vertices with the monomials $x_i^{l_i}$ and the lattice points of the edge connecting $x_i^{l_i}$ and $x_j^{l_j}$ with the monomials of the form $x_i^\alpha x_j^\beta$ where $\alpha \cdot a_i + \beta \cdot a_j = l$.
Thus we must show that the equation $\alpha \cdot a_i + \beta \cdot a_j = l$ has at least $n$ solutions $(\alpha,\beta) \in \NN^2$.
But since $l = a_0 \cdots a_n$, we have precisely $l_{ij}+1$ positive solutions, where $l_{ij}:=\frac{l}{a_ia_j}$.
Then as the $a_i$ are distinct
\[l_{ij}+1 \geq (n-1)! + 1 \geq n\]
for $n>0$ and so we're done.
\end{proof}

\begin{remark}
    Note that Corollary \ref{corollary:distinct-weights} and Corollary \ref{corollary:3-distinct-weights-VA} together prove that $\PP(1,6,10,15)$ is the first example of a weighted projective space with a non-very ample line bundle.
\end{remark}

Later, in \cite{Gubeladze:2012}, Gubeladze proved the following:

\begin{theorem}\label{theorem:gubeladze-lpe}\cite[Theorem 1.3]{Gubeladze:2012}
    Let $P \subset \RR^n$ be a polytope satisfying the $\LPE(4n(n + 1))$ property. Then $P$ is normal.
\end{theorem}

We may also characterise the $\LPE(k)$ property in rectangular simplices (and thus its consequences) in terms of a divisibility condition satisfied by its entries.

\begin{proposition}\label{proposition:lpe-characterisation}
    Let $k \in \ZZ^+$ and $\Delta(\lambda)$ be a rectangular simplex. The following are equivalent:
    \begin{enumerate}[1.]
        \item $\Delta(\lambda)$ has the $\LPE(k)$ property.
        \item For all $i, j \in [n]$ (not necessarily distinct), $\gcd(\lambda_i, \lambda_j) \geq k - 1$.
    \end{enumerate}
\end{proposition}

\begin{proof}
    The edges of $\Delta(\lambda)$ either start at $0$ and end at $\lambda_le_l$ for some $l \in [n]$, or start at $\lambda_ie_i$ and end at $\lambda_je_j$ for some $i, j \in [n]$. The edges that start at $0$ have $\lambda_l + 1 = \gcd(\lambda_l, \lambda_l) + 1$ lattice points. Let $E$ be one of the remaining edges. $E$ is fully contained in the $(i, j)$ plane, and so is lattice-equivalent to the line segment starting at $(\lambda_i, 0)$ and ending at $(0, \lambda_j)$ in $\RR^2$. It is well-known that the number of lattice points is then $\gcd(\lambda_i, \lambda_j) + 1$. 
\end{proof}

\subsection{Subsequences and extensions}\label{subsection:subsequences-and-extensions} \hfill

All statements in this subsection are still true, and definitions still make sense, after replacing ``normality'' by ``very ampleness''.

We start with an immediate consequence of the definitions.

\begin{proposition}\label{proposition:normal-implies-normal-subsequences}
    If $\Delta(\lambda)$ is normal, then $\Delta(\lambda')$ is normal for every subsequence $\lambda' \subseteq \lambda$.
\end{proposition}

We may now define two \emph{extremal} properties of rectangular simplices, the first one involving subsequences of $\lambda$.

\begin{definition}

    \begin{enumerate}[1.]
        \item We say $\Delta(\lambda)$ is \textbf{maximally non-normal} if $\mu_{\norm}(\Delta(\lambda)) = n - 1$.
        \item We say $\Delta(\lambda)$ is \textbf{sequentially non-normal} if, for every proper subsequence $\lambda' \subset \lambda$, $\Delta(\lambda')$ is normal.
    \end{enumerate}
\end{definition}

\begin{proposition}
    If $\Delta(\lambda_1, \lambda_2, \lambda_3)$ is non-normal, then it is sequentially non-normal.
\end{proposition}

A maximally non-normal rectangular simplex gives a sequentially non-normal rectangular simplex:

\begin{proposition}
    If $\Delta(\lambda)$ is maximally non-normal, then $(n - 2) \Delta(\lambda)$ is sequentially non-normal.
\end{proposition}

\begin{remark}
    In general, however, the maximal non-normality and sequential non-normality of $\Delta(\lambda)$ are not related. Indeed, it can easily be verified with the help of a computer that $\Delta(5, 11, 23, 29)$ is maximally and sequentially non-normal, $\Delta(3, 7, 11, 23)$ is maximally non-normal but not sequentially non-normal (it contains $2$ non-normal subsequences), and $\Delta(3, 7, 11, 29)$ is sequentially non-normal but not maximally non-normal.
\end{remark}

It is also interesting to consider \emph{extensions} of $\lambda$. Here, we allow $\lambda$ to have non-negative entries.

\begin{proposition}\label{proposition:normality-extensions}
    The following are equivalent:
    \begin{enumerate}[1.]
        \item $\Delta(\lambda) = \Delta(\lambda_1, \ldots, \lambda_n)$ is normal.
        \item $\Delta(\lambda, 0) = \Delta(\lambda_1, \ldots, \lambda_n, 0)$ is normal.
        \item $\Delta(\lambda, 1) = \Delta(\lambda_1, \ldots, \lambda_n, 1)$ is normal.
    \end{enumerate}
\end{proposition}

\begin{proof}
    $1. \iff 2.$ is trivial. $1. \implies 3.$ is a consequence of the fact that $S_{\Delta(\lambda, 1)} \cong S_{\Delta(\lambda)} \oplus \NN$ as semigroups (this is also \cite{BrunsGubeladze:1999}, Proposition 2.1, (b)). $3. \implies 1.$ follows from Proposition \ref{proposition:normal-implies-normal-subsequences}.
\end{proof}

Another consequence of Proposition \ref{proposition:normal-implies-normal-subsequences} is that normality indices of subsequences and extensions are related:

\begin{proposition}
    Let $\lambda' \subseteq \lambda$ be a subsequence. Then,
    \[
    \mu_{\norm}(\Delta(\lambda')) \leq \mu_{\norm}(\Delta(\lambda)) = \mu_{\norm}(\Delta(\lambda, 0)) \leq \mu_{\norm}(\Delta(\lambda, 1)).
    \]
\end{proposition}

\begin{remark}
    It can be verified computationally that very often
    \[
    \mu_{\norm}(\Delta(\lambda)) = \mu_{\norm}(\Delta(\lambda, 1)).
    \]
    Therefore, increasing the normality index cannot generally be achieved by extending $\lambda$. In other words, the existence of maximally non-normal rectangular $n$-simplices cannot be ascertained by studying rectangular $k$-simplices with $k < n$.
\end{remark}

Furthermore, almost 1-normality behaves predictably under certain extensions.

\begin{proposition} \hfill
    \begin{enumerate}[1.]
        \item $\Delta(\lambda)$ is almost 1-normal if and only if $\Delta(\lambda, 1)$ is almost 1-normal.
        \item $\Delta(\lambda_1, \ldots, \lambda_n, \lcm(\lambda_1, \ldots, \lambda_n))$ is almost 1-normal.
    \end{enumerate}
\end{proposition}

\subsection{Dilations} \hfill

In this subsection we specialise \cite[Theorem 1.1, Corollary 1.4]{HerinSchenckSmith:2006} to improve Theorem \ref{theorem:EW} for rectangular simplices.

Let $\Delta(\lambda)$ be a rectangular simplex.

\begin{proposition}\label{proposition:lattice-point-relative-interior}
    Let $\RelInt(\Delta(\lambda))$ be the relative interior of $\Delta(\lambda)$. The following are equivalent:
    \begin{enumerate}[1.]
        \item $\RelInt(\Delta(\lambda)) \cap \ZZ^n \neq \emptyset$.
        \item $(1, \stackrel{n \text{ times}}{\ldots}, 1) \in \RelInt(\Delta(\lambda))$.
        \item $\sum_{i = 1}^n \frac{1}{\lambda_i} < 1$.
    \end{enumerate}
\end{proposition}

\begin{proof}
    Everything follows from the fact that the hyperplane presentation of $\Delta(\lambda)$ is given by the inequalities
    \[
    (\delta_{ij})_{j \in [n]} \cdot x \geq 0, \quad \forall i \in [n],
    \]
    where $\delta_{ij}$ is the Kronecker delta, and
    \[
    \left ( \frac{1}{\lambda_1}, \frac{1}{\lambda_2}, \ldots, \frac{1}{\lambda_n} \right ) \cdot x \leq 1.
    \]
\end{proof}

\begin{proposition}\label{proposition:milena-bound}
    Let $r \in \ZZ_+$ and define $d(\lambda) \coloneqq \left \lfloor \sum_{i = 1}^n \frac{1}{\lambda_i} \right \rfloor$. If $r \geq n - d(\lambda)$, then $r\Delta(\lambda)$ is normal.
\end{proposition}

\begin{proof}
    Note that $d(\lambda)$ is the maximum $d \in \ZZ$ such that $\sum_{i = 1}^n \frac{1}{\lambda_i} \geq d$. The result is then a combination of \cite[Corollary 1.4]{HerinSchenckSmith:2006} and Proposition \ref{proposition:lattice-point-relative-interior}. Note also that a slightly stronger version of \cite[Corollary 1.4]{HerinSchenckSmith:2006} is actually needed, since the result is about \emph{all} multiples larger than $(n - d(\lambda))\Delta(\lambda)$, but this version is also implied by \cite[Theorem 1.1]{HerinSchenckSmith:2006}.
\end{proof}

\begin{corollary}
    Suppose $\lambda_i \neq 1$ for all $i = 1, \ldots, n$. Then,
    \[
    \sum_{i = 1}^n \frac{1}{\lambda_i} \leq \frac{n}{\min_{i \in [n]}(\lambda_i)} \leq \frac{n}{2}.
    \]
    Furthermore, both bounds are sharp, and if $\lambda$ is such that $d(\lambda)$ is equal to either bound, then $\lambda$ is normal.
\end{corollary}

\begin{corollary}
    Suppose $\lambda_i \neq 1$ for all $i = 1, \ldots, n$ and $\gcd(\lambda_i, \lambda_j) = 1$ for all $i \neq j$. Then,
    \[
    d(\lambda) = O( \log\log \left ( n(\log n + \log\log n) \right )).
    \]
\end{corollary}

\begin{proof}
    Note that $\sum_{i = 1}^n \frac{1}{\lambda_i} \leq \sum_{i = 1}^n \frac{1}{p_i}$, where $p_i$ is the $i$-th prime number. The claim then follows from classical approximation theorems, namely Mertens's second theorem \cite[(2.4)]{RosserSchoenfeld:1962} and Rosser and Schoenfeld's bound on $p_n$ \cite[(2.20)]{RosserSchoenfeld:1962}.
\end{proof}

\begin{remark}
In \cite{HenkWeismantel:1997}, Henk and Weismantel give another bound for the normality index of a polytope. In the case of a simplex $\Delta = \Conv(v_0, \dots , v_n) \subset \RR^n$, their result specialises to 
\[\munorm(\Delta) \leq n - \frac{n-1}{\det(\tilde{v}_0, \dots , \tilde{v}_n)} \enspace,\]
where $\tilde{v_i} = (1,v_i)^T$.

However, for rectangular simplices with no repeated entries (without loss of generality, see Proposition \ref{proposition:repeated-weights}), this bound is equivalent the one in Theorem \ref{theorem:EW}. Indeed, if $\Delta = \Delta(\lambda_1, \dots , \lambda_n)$ with $\lambda_i \neq \lambda_j$ for $i \neq j$, then
\[\frac{n-1}{\det(\tilde{v}_0, \dots , \tilde{v}_m)} = \frac{n-1}{\lambda_1 \cdots \lambda_n} \leq \frac{n-1}{n!} < 1\enspace .\]
\end{remark}

\subsection{Periodicity}\label{subsection:periodicity} \hfill

It is also easy to extend \cite[Theorem 1.6]{BrunsGubeladze:1999} (here, Theorem \ref{theorem:periodicity}) to show that the normality index of a rectangular simplex is also periodic.

\begin{proposition}\label{proposition:periodicity-normality-indices}
    Let $\lambda = (\lambda_1, \ldots, \lambda_n) \in \ZZ_+^n$, $i = 1, \ldots, n$ and $\ell_i = \lcm(\lambda_1, \ldots, \lambda_{i - 1}, \lambda_{i + 1}, \ldots, \lambda_n)$. Then,
    \[
    \mu_{\norm}(\Delta(\lambda)) = \mu_{\norm}(\Delta(\lambda_1, \ldots, \lambda_i + \ell_i, \ldots, \lambda_n)).
    \]
\end{proposition}

\begin{proof}
    Let $r \in \ZZ_+$. Since $\lcm(r\lambda_1, \ldots, r\lambda_{i - 1}, r\lambda_{i + 1}, \ldots, r\lambda_n) = r\ell_i$, Theorem \ref{theorem:periodicity} implies that $r\Delta(\lambda)$ is normal if and only if $\Delta(r\lambda_1, \ldots, r\lambda_i + r\ell_i, \ldots, r\lambda_n)$ is normal.
\end{proof}

\begin{corollary}\label{corollary:infinitely-many-maximally-non-normal-characterisation}
    Let $n \in \ZZ_+$. The following are equivalent:
    \begin{enumerate}[1.]
        \item There exists a maximally non-normal rectangular $n$-simplex.
        \item There exist infinitely many maximally non-normal rectangular $n$-simplices.
    \end{enumerate}
\end{corollary}

Since divisibility conditions play a role in the normality of a rectangular simplex $\Delta(\lambda)$ (see for example \cite[Proposition 2.1, Proposition 2.4]{BrunsGubeladze:1999}), it is natural to consider the case in which all entries of $\lambda$ are pairwise coprime. The following proposition and corollary give the simplest family of examples that is interesting.

\begin{proposition}
    The rectangular simplex $\Delta(2, 3, p)$, where $p \geq 5$ is a prime, is maximally non-normal if and only if $p \equiv -1 \mod 3$.
\end{proposition}

\begin{proof}
    It is easy to check that $(2, 3, 1)$ is normal and $(2, 3, 5)$ is not. By Theorem \ref{theorem:periodicity}, $(2, 3, 1 + 6a), a \geq 1$ is normal and $(2, 3, 5 + 6a), a \geq 1$ is not. Since primes are equivalent to $1$ or $5$ modulo 6, the claim follows.
\end{proof}

\begin{corollary}\label{corollary:2-3-p}
    Let $p\geq 5$ be a prime number and $\ell = \lcm(2,3,p) = 6p$.
    The line bundle $\cO(\ell)$ on weighted projective space $\PP(1,6,2p,3p)$ is very ample if and only if $p \equiv 1 \mod 3$.
\end{corollary}

Proposition \ref{proposition:periodicity-normality-indices} also implies a more refined version of Corollary \ref{corollary:infinitely-many-maximally-non-normal-characterisation} if all entries of $\lambda$ are pairwise coprime.

\begin{corollary}\label{corollary:infinitely-many-maximally-non-normal-characterisation-coprime}
    Let $n \in \ZZ_+$. The following are equivalent:
    \begin{enumerate}[1.]
        \item There exists a maximally non-normal rectangular $n$-simplex $\Delta(\lambda)$ such that the entries of $\lambda$ are pairwise coprime.
        \item There exist infinitely many maximally non-normal rectangular $n$-simplices $\Delta(\lambda)$ such that the entries of $\lambda$ are prime.
    \end{enumerate}
\end{corollary}

\begin{proof}
    2. $\implies$ 1. is trivial. For the other direction, apply Proposition \ref{proposition:periodicity-normality-indices} and Dirichlet's theorem on arithmetic progressions repeatedly.
\end{proof}

The question of whether either condition in Corollary \ref{corollary:infinitely-many-maximally-non-normal-characterisation} or \ref{corollary:infinitely-many-maximally-non-normal-characterisation-coprime} is satisfied will be answered affirmatively for all $n \in \ZZ_+$ in Section \ref{section:maximally-non-normal}.

\section{Maximally non-normal rectangular simplices}\label{section:maximally-non-normal}

Recall that a rectangular $n$-simplex $\Delta(\lambda)$ is \emph{maximally non-normal} if $\mu_{\norm}(\Delta(\lambda)) = n - 1$. In this section we give a simple algorithm that outputs, for a given $n \in \ZZ_+$, an arbitrary number of maximally non-normal rectangular $n$-simplices with distinct prime entries. To the best of our knowledge, finding even one example with unrestricted entries was previously open for high values of $n$. We rely on the notion of almost 1-normality \cite{BrunsGubeladze:1999} (here, Definition \ref{definition:a1n}).

In the following lemma, 4. $\implies$ 3. $\implies$ 2. $\iff$ 1. Furthermore, 3. is very easy to implement in a computer as a test for non-normality, and 4. will give us the main result of this section. Note also that 4. does not depend on $r \in [n - 2]$ and that 3. and 4. imply bounds on $\mu_{\mnorm}$ and $\mu_{\norm}$.

\begin{lemma}\label{lemma:non-normal-dilations-criteria}
    Let $n \geq 3$ and $r \in [n - 2]$. Let $\Delta(\lambda) = \Delta(p_1, \ldots, p_n)$ be a rectangular simplex with distinct prime entries. For $i \in [n]$, define $b_i$ to be the smallest non-negative integer equivalent to $-(\prod_{j \in [n], j \neq i}~p_j)^{-1}$ modulo $p_i$. If any of the following conditions hold, then $r\Delta(\lambda)$ is not normal:
    \begin{enumerate}[1.]
        \item $r\Delta(\lambda)$ is not almost 1-normal.
        \item The equation
        \begin{equation}\label{eq:1}
            r\left ( \prod_{i \in [n]} p_i \right ) - 1 = \sum_{i \in [n]} a_i \prod_{j \in [n], j \neq i} p_j
        \end{equation}
        has no solution $(a_1, \ldots, a_n) \in \NN^n$.
        \item There exists $k \in [n]$ such that
        \[
        \sum_{i \in [n], i \neq k} \frac{b_i}{p_i} \geq r,
        \]
        \item There exists $k \in [n]$ such that $b_i \equiv -1 \mod p_i$ for all $i \in [n], i \neq k$ and $\sum_{i \in [n], i \neq k} 1/p_i \leq 1$.
    \end{enumerate}
\end{lemma}

\begin{proof}
    \begin{enumerate}[1.]
        \item See \cite{BrunsGubeladze:1999}, where the authors show that normality implies 1-normality, which in turn implies almost 1-normality.
        \item This is Definition \ref{definition:a1n} with $L = r\prod_{i \in [n]} p_i, L_i = L/(rp_i), d = 1$.
        \item If $r\Delta(\lambda)$ is normal and $k \in [n]$, then, by 1., \eqref{eq:1} has a solution. For all $i \in [n]$,
        \[
        a_i \prod_{j \in [n], j \neq i} p_j \equiv -1 \mod p_i \iff a_i \equiv b_i, \mod p_i,
        \]
        so
        \[
        a_i = b_i + q_ip_i.
        \]
        Substituting back into \eqref{eq:1} and solving for $a_k$:
        \begin{align*}
            a_k &= \frac{r\left ( \prod_{i \in [n]} p_i \right ) - 1 - \sum_{i \in [n], i \neq k} (b_i + q_ip_i) \prod_{j \in [n], j \neq i} p_j}{\prod_{j \in [n], j \neq k} p_j} \\
            &= rp_k - \left (\sum_{i \in [n], i \neq k}q_ip_k \right ) - \frac{1 + \sum_{i \in [n], i \neq k} b_i \prod_{j \in [n], j \neq i} p_j}{\prod_{j \in [n], j \neq k} p_j} \\
            &= p_k\left(r - \left (\sum_{i \in [n], i \neq k}q_i \right )\right) - \frac{1 + \sum_{i \in [n], i \neq k} b_i \prod_{j \in [n], j \neq i} p_j}{\prod_{j \in [n], j \neq k} p_j}
        \end{align*}
        Since $a_k \geq 0$,
        \begin{align*}
            rp_k &\geq p_k\left(r - \left (\sum_{i \in [n], i \neq k}q_i \right )\right) \\
            &\geq\frac{1 + \sum_{i \in [n], i \neq k} b_i \prod_{j \in [n], j \neq i} p_j}{\prod_{j \in [n], j \neq k} p_j} \\
            &> \frac{\sum_{i \in [n], i \neq k} b_i \prod_{j \in [n], j \neq i} p_j}{\prod_{j \in [n], j \neq k} p_j},
        \end{align*}
        so
        \[
        r > \sum_{i \in [n], i \neq k} \frac{b_i}{p_i}.
        \]
        The contrapositive of the claim follows.
        \item \[
        \sum_{i \in [n], i \neq k} \frac{b_i}{p_i} = \sum_{i \in [n], i \neq k} \frac{p_i - 1}{p_i} = n - 1 - \sum_{i \in [n], i \neq k} \frac{1}{p_i} \geq n - 2 \geq r.
        \]
        Apply 3.
    \end{enumerate}
\end{proof}

\begin{remark}
    By the proof of Proposition \ref{proposition:lattice-point-relative-interior}, the second condition in Lemma \ref{lemma:non-normal-dilations-criteria} 4. is equivalent to $(1, \stackrel{n - 1 \text{ times}}{\ldots}, 1) \in \Delta(p_1, \ldots, p_{k - 1}, p_{k + 1}, \ldots, p_n)$. It is interesting to note that this is then logically consistent with Proposition \ref{proposition:milena-bound}.
\end{remark}

The following theorem is the main result of this section. Note that the condition $\sum_{i \in [n - 1]}~1/p_i~\leq~1$ is almost always satisfied.

\begin{theorem}\label{theorem:infinitely-many-maximally-non-normal}
    Let $n \geq 3$ and $P = \{p_1, \ldots, p_{n - 1}\}$ be distinct primes such that $\sum_{i \in [n - 1]}~1/p_i~\leq~1$. Then, there exist infinitely many maximally non-normal rectangular simplices of the form $\Delta(p_1, \ldots, p_{n - 1}, p_n)$ with $p_n \notin P$ prime. Furthermore, for $N \in \ZZ_+$, there exists an algorithm (Algorithm \ref{alg:1}) that outputs $N$ such simplices in finite time.
\end{theorem}

\begin{proof}
    By Lemma \ref{lemma:non-normal-dilations-criteria} 4., it is enough to find $p_n$ such that
    \[
    -1 \equiv-\left(\prod_{j \in [n], j \neq i} p_j\right)^{-1} \mod{p_i}, \quad \forall i \in [n - 1]
    \]
    and $\sum_{i \in [n - 1]} 1/p_i \leq 1$.
    Solving for $p_n$,
    \[
    p_n \equiv \alpha_i \coloneqq \left(\prod_{j \in [n - 1], j \neq i} p_j\right)^{-1} \mod{p_i}, \quad \forall i \in [n-1].
    \]
    By the Chinese Remainder Theorem, the system
    \begin{equation}\label{eq:2}
        x \equiv \alpha_i \mod{p_i}, \quad \forall i \in [n - 1]
    \end{equation}
    has a unique solution in $\ZZ/(\prod_{j \in [n - 1]}p_j)\ZZ$ lifting to a minimal $x_0 \in \NN$. The hypotheses imply that $x_0 \equiv \alpha_i \not\equiv 0 \mod{p_i}$ for all $i \in [n - 1]$, so $x_0$ and $\prod_{j \in [n - 1]}p_j$ are coprime. By Dirichlet's theorem on arithmetic progressions, the system \eqref{eq:2} has infinitely many prime solutions.
\end{proof}

\begin{algorithm}
\caption{Find Primes Defining Maximally Non-Normal Rectangular Simplices}\label{alg:1}
\begin{algorithmic}[1]
\Require A list $P$ of $n - 1$ distinct prime numbers such that $\sum_{q \in P} 1/q \leq 1$, a positive integer $N$
\Ensure A list $P'$ of $N$ distinct prime numbers such that, for each $p \in P'$, $P \cup \{p\}$ defines a maximally non-normal rectangular $n$-simplex

\State $L \gets \prod_{q \in P}q$
\State $\alpha \gets \emptyset$
\For{$q \in P$}
    \State Append inverse of $L/q$ in $\mathbb{F}_q$ to $\alpha$
\EndFor
\State $p \gets$ the solution to the system $x \equiv \alpha \mod P$
\State $P' \gets \emptyset$
\While{$\#P' < N$}
    \If{$p$ is prime}
        \State Append $p$ to $P'$
    \EndIf
    \State $p \gets p + L$
\EndWhile
\State \Return $P'$
\end{algorithmic}
\end{algorithm}

\begin{remark}
    Assuming the Generalised Riemann Hypothesis \cite{Tao:2015}, the time complexity of Algorithm \ref{alg:1} is $O(NL^2\log^2L)$, where $L = \lcm(p_1, \ldots, p_{n - 1})$.
\end{remark}

\begin{corollary}\label{corollary:2-5-7-11-619}
    The simplex $\Delta(\lambda) = \Delta(2, 5, 7, 11, 619)$ has maximal normality index.
\end{corollary}

\begin{proof}
    This follows from Algorithm \ref{alg:1} with $P = \{2, 5, 7, 11\}$ and $N = 1$, and can easily be verified using \emph{Polymake} (check for non-very-ampleness of $3\Delta(\lambda)$). 
\end{proof}

\section{Hypergraphs and the Frobenius problem}\label{section:hypergraphs}

In this section we introduce a very natural correspondence between hypergraphs and rectangular simplices that encodes the incidence structure of primes and entries of the rectangular simplex. We use it to give purely combinatorial criteria for the properties we have discussed to far and to prove almost 1-normality in many cases. We also prove that the $\LPE(n)$ property implies almost 1-normality.

\begin{definition}
    \begin{enumerate}[1.]
        \item A \textbf{family of sets} or \textbf{hypergraph} $G$ is a pair $(V, E)$ where $V$ is a non-empty, finite set and $E$ is a finite sequence of non-empty subsets of $V$ (i. e., a map $E: I \to \mathcal{P}(V) \setminus \emptyset,~~0\leq \#I < \infty$). The elements of $V$ and $E$ are called \textbf{vertices} and \textbf{edges}, respectively.
        \item Let $k \in \ZZ_+$. A hypergraph $G$ is called \textbf{$\mathbf{k}$-uniform} if every edge contains exactly $k$ vertices.
        \item The \textbf{$\mathbf{2}$-section} of a hypergraph $G = (V, E)$ is the graph $G' = (V', E')$ defined by:
        \begin{enumerate}[i.]
            \item $V' = V$.
            \item For all $v, w \in V$, $e' = \{v, w\} \in E'$ if and only if there exists $e \in E$ such that $\{v, w\} \subseteq e$.
        \end{enumerate}
        \item Let $G = (V, E)$ be a hypergraph. The \textbf{restriction} of $G$ to $V' \subset V$ is the hypergraph $G_{V'} = (V', (e \cap V': e \in E, e \cap V' \neq \emptyset))$.
        \item Let $G = (V, E: I \to \mathcal{P}(V))$ be a hypergraph. The \textbf{partial hypergraph} generated by $I' \subseteq I$ is the hypergraph $G' = (V, E|_{I'})$.
    \end{enumerate}
\end{definition}

\begin{remark}
    The previous definitions are standard. The order on $E$ is usually irrelevant (so it could also be defined as a \emph{multiset}), but since we want to label its elements, it is simpler to define it as a sequence and then quotient by an appropriate equivalence relation to get uniqueness later on. The hypergraphs we have defined are sometimes referred to as \emph{undirected}, \emph{finite} ($V$ is finite) and \emph{non-simple} or \emph{multiple} ($E$ may contain edges with only one vertex or repeated edges).
\end{remark}

\begin{definition}
    Let $G = (V, E: I \to \mathcal{P}(V))$ be a hypergraph, let $S \subseteq \ZZ_+$. An \textbf{$\mathbf{S}$-weighting} $W$ on the edges of $G$ is a map $W: I \to S$. The \textbf{total edge weight} $W(G)$ of $G$ is the product $\prod_{i \in I} W(E(i))$.
\end{definition}

\begin{lemma}\label{lemma:equivalence-relation}
    Let $V$ be a non-empty, finite set and $S \subset \ZZ_+$. Let $G_1 = (V, E_1: I_1 \to \mathcal{P}(V)), W_1: I_1 \to S$ and $G_2 = (V, E_2: I_2 \to \mathcal{P}(V)), W_2: I_2 \to S$ be two $S$-weighted hypergraphs. Define a relation $\sim_{V, S}$ by declaring that $G_1, W_1$ and $G_2, W_2$ are related if and only if there exists a bijection $f: I_1 \to I_2$ that is compatible with the edge and weighting maps (i. e. $E_1 = E_2 \circ f$ and $W_1 = W_2 \circ f$). Then, $\sim_{V, S}$ is an equivalence relation.
\end{lemma}

\begin{theorem}\label{theorem:hypergraphs-rectangular-simplices-correspondence}
    Let $\mathfrak{P} \subset \ZZ_+$ be the set of all prime numbers. There is a one-to-many correspondence between rectangular $n$-simplices and hypergraphs $G = ([n], E)$ with $\mathfrak{P}$-weighted edges such that, for all $p \in \mathfrak{P}$, the set of edges of the partial hypergraph generated by $W^{-1}(p)$ is totally ordered by inclusion. This correspondence is one-to-one after quotienting the right-hand side by $\sim_{[n], \mathfrak{P}}$.
\end{theorem}

\begin{proof}
    Left to right: draw a vertex for each index $j \in [n]$ of $\lambda$. Draw an edge $e_{p, a}$ for each $(p, a) \in \mathfrak{P} \times \mathbb{Z^+}$ such that $p^a | \lambda_j$ for some $j \in [n]$ containing all such $j \in [n]$. Set $W(e_{p, a}) = p$.

    Right to left: Start with $\lambda = (1, \ldots, 1) \in \ZZ_+^n$. For each $i \in I$ and $j \in E(i)$, multiply $\lambda_j$ by $p = W(E(i))$.

    Both procedures are clearly well-defined and inverses of each other.
\end{proof}

\begin{figure}[ht]
    \centering
    \begin{subfigure}[t]{0.49\textwidth}
        \centering
        \includesvg[height=1.25in]{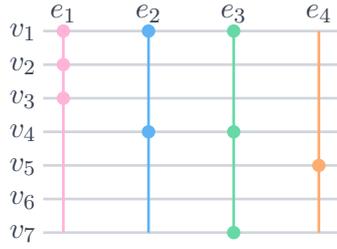}
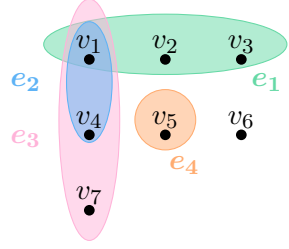
        \caption{Incidence matrix representation of $G$.}
    \end{subfigure}
    \hfill
    \begin{subfigure}[t]{0.49\textwidth}
        \centering
        \begin{tikzpicture}
          % edges
          \draw [fill=mint-green, fill opacity = 0.5, draw=mint-green, text=mint-green, text opacity = 1] (0,1) ellipse [x radius=1.6, y radius=0.4]
            node [below right, xshift=1cm, yshift=-0.25cm] {$\bm{e_1}$};
          \draw [fill=pastel-pink, fill opacity = 0.5, draw=pastel-pink, text=pastel-pink, text opacity = 1] (-1,0) ellipse [x radius=0.4, y radius=1.6]
            node [below left, xshift=-0.5cm, yshift=0cm] {$\bm{e_3}$};
          \draw [fill=soft-blue, fill opacity = 0.5, draw=soft-blue, text=soft-blue, text opacity = 1] (-1,0.5) ellipse [x radius=0.3, y radius=0.8]
            node [left, xshift=-0.5cm, yshift=0cm] {$\bm{e_2}$};
          \draw [fill=peach, fill opacity = 0.5, draw=peach, text=peach, text opacity = 1] (0,0) ellipse [x radius=0.4, y radius=0.4]
            node [below, xshift=0.25cm, yshift=-0.35cm] {$\bm{e_4}$};

          % vertices
          \node (A) at (-1,1) {$v_1$};
          \node (B) at (0,1) {$v_2$};
          \node (C) at (1,1) {$v_3$};
          \node (D) at (-1,0) {$v_4$};
          \node (E) at (0,0) {$v_5$};
          \node (F) at (1,0) {$v_6$};
          \node (G) at (-1,-1) {$v_7$};

          \foreach \v in {A,B,C,D,E,F,G} {
            \fill (\v) circle [radius=2pt, xshift=0cm, yshift=-0.2cm];
          }
        \end{tikzpicture}
        \caption{Venn diagram representation of $G$.}
    \end{subfigure}
    \caption{Two representations of a hypergraph $G$.}
    \label{figure:hypergraph-wikipedia}
\end{figure}

\begin{example}
    The hypergraph $G$ in Figure \ref{figure:hypergraph-wikipedia} corresponds to rectangular simplices of the form
    \[
    \Delta(\lambda) = \Delta(p_1p_2p_3,~~p_1,~~p_1,~~p_2p_3,~~p_4,~~1,~~p_3),
    \]
    with $p_i \in \mathfrak{P}$ and $p_i \neq p_j$ for $i \neq j, \{i, j\} \neq \{2, 3\}$. In other words, $p_2$ and $p_3$ may be equal -- this is a consequence of the fact that $e_2 \subset e_3$.
\end{example}

The following are numerical consequences involving the weighting on the edges of the hypergraph.

\begin{proposition}\label{proposition:conditions-depending-on-weighting}
    Let $\Delta(\lambda) = \Delta(\lambda_1, \ldots, \lambda_n)$ be a rectangular simplex and let $G$ be the corresponding hypergraph with weighting map $W$.
    \begin{enumerate}[1.]
        \item The greatest common divisor and least common multiple of a subsequence $(\lambda_{j_1}, \ldots, \lambda_{j_m}) \subset (\lambda_1, \ldots, \lambda_n)$ of the entries of $\lambda$ satisfy
        \[
        \gcd(\lambda_{j_1}, \ldots, \lambda_{j_m}) = \prod_{\substack{i \in I\\\{j_1, \ldots, j_m\} \subseteq E(i)}} W(i)
        \]
        and
        \[
        \lcm(\lambda_{j_1}, \ldots, \lambda_{j_m}) = \prod_{\substack{i \in I\\\exists j \in \{j_1, \ldots, j_m\}:~~j \in E(i)}} W(i) = W(G_{\{j_1, \ldots, j_m\}}).
        \]
        \item If $I' \subseteq I$ is such that $E(i)$ covers $V$ for all $i \in I'$, then $W((V, E_{I'})) | \lambda$. Conversely, if $d | \lambda$, then there exists a corresponding $I' \subseteq I$. In particular, $\gcd(\lambda_1, \ldots, \lambda_n) = \prod_{\substack{i \in I\\V \subset E(i)}} W(i)$ (this also follows from 1.)
        \item Let $k \in \ZZ_+$. The rectangular simplex $\lambda$ has the $\LPE(k)$ property if and only if the corresponding hypergraph $G$ satisfies that
        \[
        \prod_{\substack{i \in I\\\{v, w\} \subseteq E(i)}} W(i) \geq k - 1, \quad \forall v, w \in V.
        \]
    \end{enumerate}
\end{proposition}

The following conditions surprisingly \emph{do not} depend on the weighting on the edges of the hypergraph, and are largely consequences of Proposition \ref{proposition:conditions-depending-on-weighting} and \cite{BrunsGubeladze:1999}.

\begin{proposition}\label{proposition:conditions-not-depending-on-weighting}
    Let $\Delta(\lambda) = \Delta(\lambda_1, \ldots, \lambda_n)$ be a rectangular simplex and let $G$ be the corresponding hypergraph. Then:
    \begin{enumerate}[1.]
        \item If $G$ has at most two vertices, then $\Delta(\lambda)$ is normal.
        \item If $G$ has at most three vertices, then $\Delta(\lambda)$ is normal if and only if $\Delta(\lambda)$ is very ample.
        \item Let $(i, j) \in V^2$. Then, $\lambda_i | \lambda_j$ if and only if the edges of $G_{\{i\}}$ are a subsequence of the edges of $G_{\{j\}}$ (Condition 1). In particular, if $G$ has three vertices and there exists a pair $(i, j) \in V^2$ satisfying Condition 1, then $\Delta(\lambda)$ is normal. Also, if $(n - 2, n)$ and $(i, i + 1)$ for all $i \in [n - 2]$ satisfy Condition 1, then $\Delta(\lambda)$ is normal and koszul.
        \item The entries of $\lambda$ are setwise coprime if and only if no edge of $G$ covers $V$.
        \item If $G$ is $1$-uniform, then the entries of $\lambda$ are pairwise coprime. If $G$ is $n$-uniform, then $\lambda$ is of the form $\lambda = (\lambda', \ldots, \lambda')$ for some $\lambda' \in \ZZ_+$ (and so $\Delta(\lambda)$ is normal).
        \item If $\# \{i \in I: V \subseteq E(i)\} \geq \log_2(n - 1)$, then $\Delta(\lambda)$ is normal, and if $\# \{i \in I: V \subseteq E(i)\} \geq \log_2(n)$, then $\Delta(\lambda)$ is koszul.
        \item If the $2$-section of $G$ is the complete graph $K_n$, then $\Delta(\lambda)$ has the $\LPE(3)$ property.
        \item If $\min_{(v, w) \in V^2} (\# \{i \in I: \{v, w\} \subseteq E(i)\}) \geq \log_2(k - 1)$, then $\Delta(\lambda)$ has the $\LPE(k)$ property. In particular, $\geq \log_2(n - 1)$ implies very ampleness and $\geq \log_2(4n(n + 1) - 1)$ implies normality.
        \item $\Delta(\lambda)$ is normal if and only if the rectangular simplex corresponding to $G_{\cup_{i \in I}E(i)}$ is normal. In particular, if $\# \cup_{i \in I} E(i) \leq 2$, then $\Delta(\lambda)$ is normal, and if $\# \cup_{i \in I} E(i) \leq 3$, then $\Delta(\lambda)$ is normal if and only if $\Delta(\lambda)$ is very ample.
    \end{enumerate}
\end{proposition}

\begin{proof}
    \begin{enumerate}[1.]
        \item All convex lattice polygons are normal.
        \item This follows from \cite{Ogata:2005}.
        \item This is equivalent to \cite[Proposition 2.1 (c)]{BrunsGubeladze:1999} and \cite[Proposition 2.4 (b)]{BrunsGubeladze:1999}.
        \item[4., 5.] These follow from Proposition \ref{proposition:conditions-depending-on-weighting} 2.
        \item[6.] This follows from Proposition \ref{proposition:conditions-depending-on-weighting} 2. and \cite[Proposition 2.4 (c)]{BrunsGubeladze:1999}.
        \item[7.] This follows from Proposition \ref{proposition:conditions-depending-on-weighting} 1.
        \item[8.] This follows from Proposition \ref{proposition:conditions-depending-on-weighting} 1., Theorem \ref{theorem:Payne-LPE} and \cite[Theorem 1.3]{Gubeladze:2012} (here, Theorem \ref{theorem:gubeladze-lpe}).
        \item[9.] This follows from Proposition \ref{proposition:normality-extensions} and 1., 2.
    \end{enumerate}
\end{proof}

\begin{definition}
    Let $a = (a_1, \ldots, a_n) \in \ZZ_+^n$ such that $\gcd(a_1, \ldots, a_n) = 1$. The \textbf{Frobenius number} $F(a)$ is the largest integer that cannot be represented as a linear combination of the $a_i$ with non-negative integer coefficients.
\end{definition}

\begin{proposition}\label{proposition:bounds}
    Let $a = (a_1, \ldots, a_n) \in \ZZ_+^n$ such that $\gcd(a_1, \ldots, a_n) = 1$ and $a_1~\leq~a_2~\leq~\cdots~\leq~a_n$. The following expressions are non-strict upper bounds for $F(a)$.
    \begin{enumerate}[1.]
        \item Erdős-Graham, \cite[Theorem 1]{ErdosGraham:1972}:
        \[
        2a_{n - 1} \left  \lfloor \frac{a_n}{n} \right \rfloor - a_n.
        \]
        \item Selmer, \cite{Selmer:1977}:
        \[
        2a_n \left  \lfloor \frac{a_1}{n} \right \rfloor - a_1.
        \]
        \item Brauer, \cite{Brauer:1942}:
        \[
        a_1 \frac{d_1}{d_2} + a_3 \frac{d_2}{d_3} + \cdots + a_n \frac{d_{n - 1}}{d_n} - \sum_{i = 1}^n a_n,
        \]
        where $d_i = \gcd(a_1, a_2, \ldots, a_i)$.
    \end{enumerate}
\end{proposition}

\begin{lemma}\label{lemma:d}
    Let $a = (a_1, \ldots, a_n) \in \ZZ_+^n$, let $L = \lcm(a_1, \ldots, a_n)$. Then,
    \[
    d = \gcd{ \left ( \frac{L}{a_1}, \frac{L}{a_2}, \ldots, \frac{L}{a_n} \right )} = 1.
    \]
\end{lemma}

\begin{proof}
    Let $p \in \ZZ_+$ be a prime dividing $d$. Since it appears in $L$ with exponent $\geq 1$, it appears in some $a_j$ with exponent $\geq 1$. Let $i \in [n]$ be such that $v_p(a_i) = \max_j v_p(a_j) = v_p(L)$. But then $p$ does not divide $L/a_i$, a contradiction. 
\end{proof}

\begin{proposition}\label{proposition:a1n-characterisation}
    Let $L = \lcm(\lambda_1, \ldots, \lambda_n), L_i = \frac{L}{\lambda_i}, d = \gcd(L_1, \ldots, L_n)$. The following are equivalent:
    \begin{enumerate}[1.]
        \item The rectangular simplex $\Delta(\lambda)$ is A1N.
        \item $L - d$ is a linear combination of the $L_i$ with non-negative integer coefficients. This is the definition in \cite{BrunsGubeladze:1999} (here, Definition \ref{definition:a1n}).
        \item $L - 1$ is a linear combination of the $L_i$ with non-negative integer coefficients.
    \end{enumerate}
    Furthermore, if $L - 1 > F(L_1, \ldots, L_n)$, then the rectangular simplex $\lambda$ is A1N.
\end{proposition}

\begin{proof}
    The equivalence follows from Lemma \ref{lemma:d}. The implication follows from the definition of the Frobenius number.
\end{proof}

\begin{figure*}[ht]
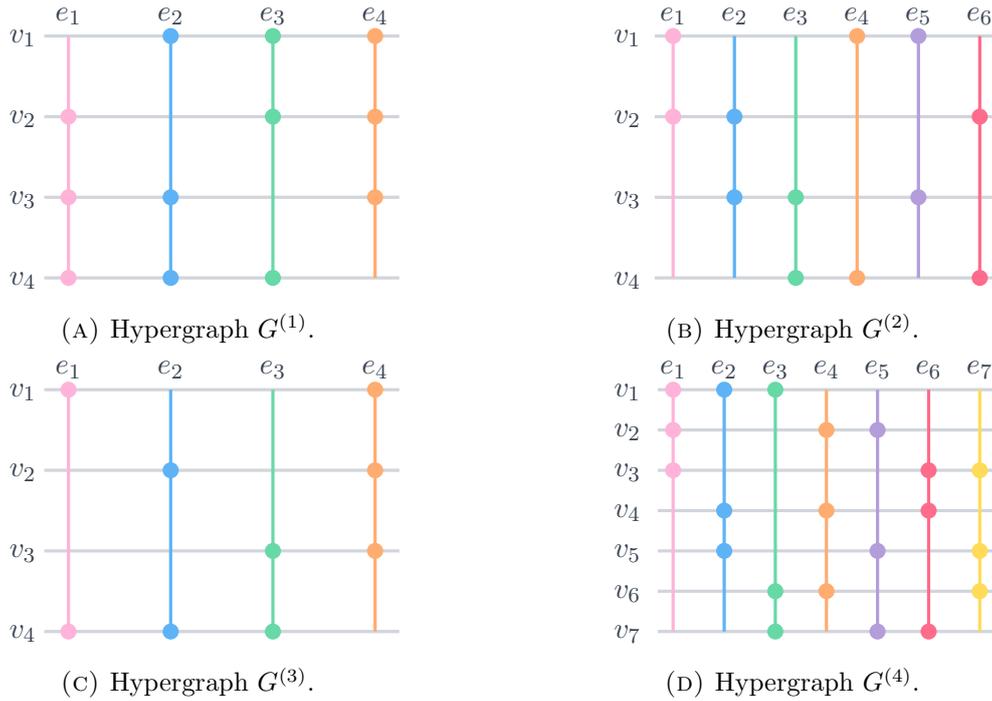

        \centering
        \begin{subfigure}[t]{0.49\textwidth}
            \centering
            \includesvg[height=1.5in]{hypergraph-example-1.svg}
            \caption{Hypergraph $G^{(1)}$.}
        \end{subfigure}
        \hfill
        \begin{subfigure}[t]{0.49\textwidth}
            \centering
            \includesvg[height=1.5in]{hypergraph-example-2.svg}
            \caption{Hypergraph $G^{(2)}$.}
        \end{subfigure}
        \medskip
        \begin{subfigure}[t]{0.49\textwidth}
            \centering
            \includesvg[height=1.5in]{hypergraph-example-3.svg}
            \caption{Hypergraph $G^{(3)}$.}
        \end{subfigure}
        \hfill
        \begin{subfigure}[t]{0.49\textwidth}
            \centering
            \includesvg[height=1.5in]{hypergraph-example-4.svg}
            \caption{Hypergraph $G^{(4)}$.}
        \end{subfigure}
        \caption{Some example hypergraphs. Hypergraph $G_2$ corresponds to the complete graph $K_4$, and hypergraph $G_4$ corresponds to the \emph{Fano plane}.}\label{figure:example-hypergraphs}
    \end{figure*}

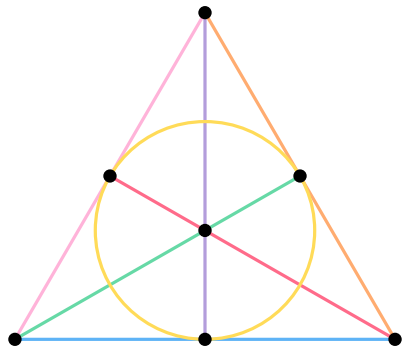
\begin{figure}[ht]
    \centering
    \begin{tikzpicture}[scale=1.25]
      % coordinates for an equilateral triangle of side 4
      \coordinate (v1) at (-2,0);
      \coordinate (v4) at (2,0);
      \coordinate (v2) at (0,3.4641);       % 2*sqrt(3)
    
      % midpoints of edges
      \coordinate (v5) at (0,0);
      \coordinate (v6) at (1,1.7321);      % sqrt(3)/2*2
      \coordinate (v3) at (-1,1.7321);
    
      % incenter/centroid
      \coordinate (v7) at (0,1.1547);        % (2*sqrt(3))/3
    
      % edges
      \draw[very thick, pastel-pink] (v1)--(v2);
      \draw[very thick, soft-blue] (v1)--(v4);
      \draw[very thick, mint-green] (v1)--(v6);
      \draw[very thick, peach] (v2)--(v4);
      \draw[very thick, lavender] (v2)--(v5);
      \draw[very thick, coral] (v3)--(v4);
    
      % circle
      \draw[very thick, yellow] (v7) circle[radius=1.1547];
    
      % vertices
      \fill (v1) circle[radius=2pt];
      \fill (v2) circle[radius=2pt];
      \fill (v3) circle[radius=2pt];
      \fill (v4) circle[radius=2pt];
      \fill (v5) circle[radius=2pt];
      \fill (v6) circle[radius=2pt];
      \fill (v7) circle[radius=2pt];
    \end{tikzpicture}
    \caption{The \emph{Fano plane}.}
    \label{figure:fano-plane}
\end{figure}

\begin{example}\label{example:hypergraphs-lpe-a1n}
    The hypergraphs $G_1, \ldots, G_4$ in Figure \ref{figure:example-hypergraphs} correspond respectively to rectangular simplices of the form
    \begin{align*}
        \Delta(\lambda^{(1)}) &= \Delta(p_2p_3p_4,~~p_1p_3p_4,~~p_1p_2p_4,~~p_1p_2p_3), \\
        \Delta(\lambda^{(2)}) &= \Delta(p_1p_4p_5,~~p_1p_2p_6,~~p_2p_3p_5,~~p_3p_4p_6), \\
        \Delta(\lambda^{(3)}) &= \Delta(p_1p_4,~~p_2p_4,~~p_3p_4,~~p_1p_2p_3), \\
        \Delta(\lambda^{(4)}) &= \Delta(p_1p_2p_3,~~p_1p_4p_5,~~p_1p_6p_7,~~p_2p_4p_6,~~p_2p_5p_7,~~p_3p_4p_7,~~p_3p_5p_6),
    \end{align*}
    with $p_i \in \mathfrak{P}$ and $p_i \neq p_j$ for $i \neq j$. Hypergraph $G_2$ corresponds to the complete graph $K_4$. Hypergraph $G_4$ corresponds to the \emph{Fano plane} (see Figure \ref{figure:fano-plane}). Hypergraphs $G_1$, $G_2$ and $G_3$ are easily generalisable to $n \geq 5$.

    The rectangular simplex $\Delta(\lambda^{(1)})$ has the $\LPE((\min_{i \neq j} p_ip_j) + 1)$ property. The rectangular simplices $\Delta(\lambda^{(2)}), \Delta(\lambda^{(3)}), \Delta(\lambda^{(4)})$ have the $\LPE((\min_{i} p_i) + 1)$ property.
    
    $\Delta(\lambda^{(1)})$ and $\Delta(\lambda^{(2)})$ are A1N due to Selmer's bound. $\Delta(\lambda^{(3)})$ is A1N due to Brauer's bound. The A1N property of $\Delta(\lambda^{(4)})$ can also be easily characterised using Brauer's bound.
\end{example}

\begin{theorem}\label{theorem:lpe-implies-a1n}
    Let $\Delta(\lambda) = \Delta(\lambda_1, \ldots, \lambda_n)$ be a rectangular simplex satisfying the $\LPE(n)$ property. Then $\Delta(\lambda)$ is A1N.
\end{theorem}

\begin{proof}
    Lemma \ref{lemma:d} implies that $d = \gcd(L_1, \ldots, L_n) = 1$. Without loss of generality, we may assume $\lambda_1 \geq \lambda_2 \geq \cdots \geq \lambda_n$, so $L_1 \leq L_2 \leq \cdots \leq L_n$. Then, by Brauer's bound (Proposition \ref{proposition:bounds}) and Proposition \ref{proposition:a1n-characterisation}, if
    \begin{equation}\label{equation:brauer-a1n}
        L - 1 > L_2 \frac{d_1}{d_2} + L_3 \frac{d_2}{d_3} + \cdots + L_n \frac{d_{n - 1}}{d_n} - \sum_{i = 1}^n L_i,
    \end{equation}
    where $d_i = \gcd(L_1, \ldots, L_i)$, then the rectangular simplex $\Delta(\lambda)$ is A1N. For $i \in [n - 1]$, set $L'_i = \lcm(\lambda_1, \ldots, \lambda_i)$. Then,
    \begin{align*}
        L_{i + 1} \frac{d_i}{d_{i + 1}} &= \frac{L}{\lambda_{i + 1}} \frac{\gcd \left ( \frac{L}{\lambda_1}, \ldots, \frac{L}{\lambda_i} \right )}{\gcd \left ( \frac{L}{\lambda_1}, \ldots, \frac{L}{\lambda_{i + 1}} \right )} \\
        &= \frac{L}{\lambda_{i + 1}} \frac{\gcd \left ( \frac{L \frac{L'_i}{L'_i}}{\lambda_1}, \ldots, \frac{L\frac{L'_i}{L'_i}}{\lambda_i} \right )}{\gcd \left ( \frac{L\frac{L'_{i + 1}}{L'_{i + 1}}}{\lambda_1}, \ldots, \frac{L\frac{L'_{i + 1}}{L'_{i + 1}}}{\lambda_{i + 1}} \right )} \\
        &= \frac{L}{\lambda_{i + 1}} \frac{\frac{L}{L'_i}}{\frac{L}{L'_{i + 1}}} \\
        &= \frac{L L'_{i + 1}}{\lambda_{i + 1} L'_i} \\
        &= \frac{L L'_{i + 1}}{\lcm(\lambda_{i + 1}, L'_i)\gcd(\lambda_{i + 1}, L'_i)} \\
        &= \frac{L L'_{i + 1}}{L'_{i + 1} \lcm(\gcd(\lambda_{i + 1}, \lambda_1), \gcd(\lambda_{i + 1}, \lambda_2), \ldots, \gcd(\lambda_{i + 1}, \lambda_{i}))} \\
        &= \frac{L}{\lcm_{j \in [i]}(\gcd(\lambda_{i + 1}, \lambda_j))},
    \end{align*}
    where we have used Lemma \ref{lemma:d} again to get the third equality, and the fact that the positive integers form a distributive lattice to get the sixth equality. Therefore, Inequality \ref{equation:brauer-a1n} is equivalent to
    \[
        1 - \frac{1}{L} > \sum_{i = 1}^{n - 1} \frac{1}{\lcm_{j \in [i]}(\gcd(\lambda_{i + 1}, \lambda_j))} - \sum_{i = 1}^n \frac{1}{\lambda_i}.
    \]
    Since $\lambda_1 < L$, it is sufficient that
    \[
    1 \geq \frac{(n - 1)}{\min_{i \in [n - 1]}(\lcm_{j \in [i]}(\gcd(\lambda_{i + 1}. \lambda_j)))}.
    \]
    But the fact that $\Delta(\lambda)$ has the $\LPE(n)$ property, along with Proposition \ref{proposition:lpe-characterisation}, imply that $\gcd(\lambda_{i + 1}, \lambda_j) \geq n - 1$ for all $i, j$. The result follows.
\end{proof}

\begin{remark}
    The previous proof shows that the value $k = n$ in Theorem \ref{theorem:lpe-implies-a1n} is sharp. For example, by Corollary \ref{corollary:2-5-7-11-619}, the rectangular simplex $\Delta(\lambda) = \Delta(3 \cdot 2,~~3 \cdot 5,~~3 \cdot 7,~~3 \cdot 11,~~3 \cdot 619)$ is not A1N, but by Proposition \ref{proposition:lpe-characterisation}, it has the $\LPE(4)$ property.
\end{remark}

\section*{Further directions}
We would like to finish by repeating one of our motivating questions:

\begin{question}
    Is there a rectangular simplex that is very ample but not normal?
\end{question}

Theorem \ref{theorem:lpe-implies-a1n} implies that the $\LPE(n)$ property and almost 1-normality cannot easily be used to answer this question affirmatively. However, the methods in this section can still be used to find families of rectangular $n$-simplices satisfying the $\LPE(k)$ property for $n \leq k < 4n(n + 1)$. These simplices may be very ample and almost 1-normal but not normal (see Theorems \ref{theorem:Payne-LPE} and \ref{theorem:gubeladze-lpe}). We would be interested in any future research in this direction.

\bibliographystyle{amsalpha}
\bibliography{bib}

\newcommand{\etalchar}[1]{$^{#1}$}
\providecommand{\bysame}{\leavevmode\hbox to3em{\hrulefill}\thinspace}
\providecommand{\MR}{\relax\ifhmode\unskip\space\fi MR }
\providecommand{\MRhref}[2]{%
  \href{http://www.ams.org/mathscinet-getitem?mr=#1}{#2}
}
\providecommand{\href}[2]{#2}
\begin{thebibliography}{BDH{\etalchar{+}}24}

\bibitem[BDH{\etalchar{+}}24]{BraunDavisHanelySolus:2024}
Benjamin Braun, Robert Davis, Derek Hanely, Morgan Lane, and Liam Solus, \emph{The integer decomposition property and weighted projective space simplices}, Integers \textbf{24} (2024), Paper No. A60, 59. \MR{4762303}

\bibitem[BDS18]{BraunDavisSolus:2018}
Benjamin Braun, Robert Davis, and Liam Solus, \emph{Detecting the integer decomposition property and {E}hrhart unimodality in reflexive simplices}, Adv. in Appl. Math. \textbf{100} (2018), 122--142. \MR{3835192}

\bibitem[BG99]{BrunsGubeladze:1999}
Winfried Bruns and Joseph Gubeladze, \emph{Rectangular simplicial semigroups}, Commutative algebra, algebraic geometry, and computational methods ({H}anoi, 1996), Springer, Singapore, 1999, pp.~201--213. \MR{1714858}

\bibitem[BG09]{BrunsGubeladze:2009}
\bysame, \emph{Polytopes, rings, and {$K$}-theory}, Springer Monographs in Mathematics, Springer, Dordrecht, 2009. \MR{2508056}

\bibitem[BGT97]{BrunsGubeladzeTrung:1997}
Winfried Bruns, Joseph Gubeladze, and Ng\^o{}~Vi\^et Trung, \emph{Normal polytopes, triangulations, and {K}oszul algebras}, J. Reine Angew. Math. \textbf{485} (1997), 123--160. \MR{1442191}

\bibitem[Bra42]{Brauer:1942}
Alfred Brauer, \emph{On a problem of partitions}, Amer. J. Math. \textbf{64} (1942), 299--312. \MR{6196}

\bibitem[CGMe20]{ColomenarejoGaluppiMichalek:2020}
Laura Colmenarejo, Francesco Galuppi, and Mateusz Micha\l~ek, \emph{Toric geometry of path signature varieties}, Adv. in Appl. Math. \textbf{121} (2020), 102102, 35. \MR{4140555}

\bibitem[CHHH14]{CoxHasseHibiHigashitani:2014}
David~A. Cox, Christian Haase, Takayuki Hibi, and Akihiro Higashitani, \emph{Integer decomposition property of dilated polytopes}, Electron. J. Combin. \textbf{21} (2014), no.~4, Paper 4.28, 17. \MR{3292265}

\bibitem[CHM07]{Oberwolfach:2007}
Takayuki~Hibi Christian~Haase and Diane Maclagan, \emph{Mini-workshop: Projective normality of smooth toric varieties}, 2007, Abstracts from the mini- workshop held August 12–18, 2007. Organized by Christian Haase, Takayuki Hibi and Diane Maclagan, pp.~2283--2320.

\bibitem[CLS11]{CoxLittleSchenck:2011}
David~A. Cox, John~B. Little, and Henry~K. Schenck, \emph{Toric varieties}, Graduate Studies in Mathematics, vol. 124, American Mathematical Society, Providence, RI, 2011. \MR{2810322}

\bibitem[EoG72]{ErdosGraham:1972}
P.~Erd\H~os and R.~L. Graham, \emph{On a linear diophantine problem of {F}robenius}, Acta Arith. \textbf{21} (1972), 399--408. \MR{311565}

\bibitem[EW91]{EwaldWessels:1991}
G\"unter Ewald and Uwe Wessels, \emph{On the ampleness of invertible sheaves in complete projective toric varieties}, Results Math. \textbf{19} (1991), no.~3-4, 275--278. \MR{1100674}

\bibitem[Gub12]{Gubeladze:2012}
Joseph Gubeladze, \emph{Convex normality of rational polytopes with long edges}, Adv. Math. \textbf{230} (2012), no.~1, 372--389. \MR{2900547}

\bibitem[Her06]{Hering:Thesis}
Milena~S. Hering, \emph{Syzygies of toric varieties}, ProQuest LLC, Ann Arbor, MI, 2006, Thesis (Ph.D.)--University of Michigan. \MR{2708466}

\bibitem[HSS06]{HerinSchenckSmith:2006}
Milena Hering, Hal Schenck, and Gregory~G. Smith, \emph{Syzygies, multigraded regularity and toric varieties}, Compos. Math. \textbf{142} (2006), no.~6, 1499--1506. \MR{2278757}

\bibitem[HW97]{HenkWeismantel:1997}
Martin Henk and Robert Weismantel, \emph{The height of minimal {H}ilbert bases}, Results Math. \textbf{32} (1997), no.~3-4, 298--303. \MR{1487601}

\bibitem[Kas09]{Kasprzyk:2009}
Alexander~M. Kasprzyk, \emph{Bounds on fake weighted projective space}, Kodai Math. J. \textbf{32} (2009), no.~2, 197--208. \MR{2549542}

\bibitem[MP25]{MullerPaemurru:2025}
Stevell Muller and Erik Paemurru, \emph{Very ample sheaves on weighted projective spaces and weighted blowups}, 2025.

\bibitem[Oga05]{Ogata:2005}
Shoetsu Ogata, \emph{{$k$}-normality of weighted projective spaces}, Kodai Math. J. \textbf{28} (2005), no.~3, 519--524. \MR{2194542}

\bibitem[Pay06]{Payne:2006}
Sam Payne, \emph{Fujita's very ampleness conjecture for singular toric varieties}, Tohoku Math. J. (2) \textbf{58} (2006), no.~3, 447--459. \MR{2273280}

\bibitem[RS62]{RosserSchoenfeld:1962}
J.~Barkley Rosser and Lowell Schoenfeld, \emph{Approximate formulas for some functions of prime numbers}, Illinois J. Math. \textbf{6} (1962), 64--94. \MR{137689}

\bibitem[Sel77]{Selmer:1977}
Ernst~S. Selmer, \emph{On the linear diophantine problem of frobenius}, Journal für die reine und angewandte Mathematik \textbf{1977} (1977), no.~293, 1--17 (eng).

\bibitem[Tao15]{Tao:2015}
Terence Tao, \emph{{254A, Notes 7: Linnik’s theorem on primes in arithmetic progressions}}, \url{https://terrytao.wordpress.com/2015/02/22/254a-notes-7-linniks-theorem-on-primes-in-arithmetic-progressions/}, February 2015, Accessed: 2025-10-01.

\end{thebibliography}

\end{document}